\documentclass[12pt]{amsart}
\usepackage[T1]{fontenc}
\usepackage[latin9]{inputenc}
\usepackage{geometry}
\usepackage{verbatim}
\usepackage{enumitem}
\usepackage{amstext}
\usepackage{amssymb}
\usepackage{stmaryrd}
\usepackage[unicode=true,pdfusetitle,
 bookmarks=true,bookmarksnumbered=false,bookmarksopen=false,
 breaklinks=false,pdfborder={0 0 1},backref=false,colorlinks=false]
 {hyperref}

\makeatletter

\usepackage{amsthm}

\usepackage{rotating}

\usepackage{amsfonts}\usepackage{amscd}\usepackage[all]{xy}\usepackage{color}\usepackage{tikz}\usepackage{amsthm}

\numberwithin{equation}{section}

\newtheorem{thm}{Theorem}[section]
\newtheorem*{thm*}{Theorem}
\newtheorem{cor}[thm]{Corollary}
\newtheorem*{cor*}{Corollary}
\newtheorem{lem}[thm]{Lemma}
\newtheorem*{lem*}{Lemma}
\newtheorem{prop}[thm]{Proposition}
\newtheorem*{prop*}{Proposition}

\newtheorem*{conjecture*}{Conjecture}

\newtheorem*{fact*}{Conjecture}

\newtheorem*{criterion*}{Criterion}

\newtheorem*{algorithm*}{Algorithm}

\newtheorem*{ax*}{Axiom}

\newtheorem*{assumption*}{Assumption}

\newtheorem*{question*}{Question}

\theoremstyle{remark}

\newtheorem{rem}[thm]{Remark}
\newtheorem*{rem*}{Remark}
\newtheorem{rems}[thm]{Remarks}
\newtheorem*{rems*}{Remarks}

\newtheorem*{claim*}{Claim}

\newtheorem*{exercise*}{Exercise}

\newtheorem*{note*}{Note}
\newtheorem{notation}[thm]{Notation}
\newtheorem*{notation*}{Notation}

\newtheorem*{summary*}{Summary}

\newtheorem*{acknowledgement*}{Acknowledgement}

\newtheorem*{conclusion*}{Conclusion}

\theoremstyle{definition}

\newtheorem{defn}[thm]{Definition}
\newtheorem*{defn*}{Definition}
\newtheorem{example}[thm]{Example}
\newtheorem*{example*}{Example}
\newtheorem{examples}[thm]{Examples}
\newtheorem*{examples*}{Examples}

\newtheorem*{problem*}{Problem}

\newtheorem*{xca*}{Exercise}

\newtheorem*{xcas*}{Exercises}

\newtheorem*{condition*}{Condition}

\author[Cluckers]
{Raf Cluckers}
\address{Univ.~Lille, CNRS, UMR 8524 - Laboratoire Paul Painlev\'e, F-59000 Lille, France, and,
KU Leuven, Department of Mathematics, B-3001 Leuven, Belgium}
\email{Raf.Cluckers@univ-lille.fr}
\urladdr{http://rcluckers.perso.math.cnrs.fr/}

\author[Comte]{Georges Comte}
\address{Universit\'e Savoie Mont Blanc, LAMA,
CNRS UMR 5127,
F-73000 Chamb\'ery, France}
\email{georges.comte@univ-smb.fr}
\urladdr{https://georgescomte.perso.math.cnrs.fr/}

\author[Rolin]{Jean-Philippe Rolin}
\address{Institut de mathématiques de Bourgogne, UMR 5584 CNRS, Universit\'e de Bourgogne,
F-21000 Dijon, France}
\email{jean-philippe.rolin@u-bourgogne.fr}
\urladdr{http://rolin.perso.math.cnrs.fr/}

\author[Servi]{Tamara Servi}
\address{Institut de Math\'ematiques de Jussieu -- Paris Rive Gauche \\
	Universit\'{e} Paris Cit\'{e} and Sorbonne Universit\'{e}, CNRS, IMJ-PRG, F-75013 Paris, France}
\email{tamara.servi@imj-prg.fr}
\urladdr{http://www.logique.jussieu.fr/~servi/index.html}

\makeatother

\begin{document}
\title{Mellin transforms of power-constructible functions}
\begin{abstract}
We consider several systems of algebras of real- and complex-valued
functions, which appear in o-minimal geometry and related geometrically
tame contexts. For each such system, we prove its stability under
parametric integration and we study the asymptotics of the functions
as well as the nature of their parametric Mellin transforms. 

\tableofcontents{}
\end{abstract}

\maketitle
\begin{acknowledgement*}
The IMB receives support from the EIPHI Graduate School (contract
ANR-17-EURE-0002). The third and fourth authors would like to thank
the Fields Institute for Research in Mathematical Sciences for its
hospitality and financial support, as part of this work was done while
at its Thematic Program on Tame Geometry and Applications in 2022.
The first author was partially supported by the Labex CEMPI (ANR-11-LABX-0007-01).
\end{acknowledgement*}

\section{Introduction\label{sec:Introduction}}

In this work we pursue the investigation, started in \cite{ccmrs:integration-oscillatory},
of certain parametric integral transforms from the point of view of
tame analysis (in \cite{ccmrs:integration-oscillatory} we studied
the parametric Fourier transform, here we consider the parametric
Mellin transform, and in a forthcoming paper we analyze the combined
action of these two operators on certain collections of tame functions).

The study of parametric integrals of functions belonging to a given
tame class arises from the question of the nature of the volume of
the fibres $X_{y}$ of a tame family $(X_{y})_{y\in Y}$. More precisely,
describing the locus of integrability is a counterpart to establishing
the nature of the set of points $y$ of $Y$ for which $X_{y}$ has
finite volume. The volumes of globally subanalytic sets have been
studied in \cite{lr:int,clr}, where it is proven that, for a globally
subanalytic set $X\subseteq\mathbb{R}^{n+m}$ such that the fibres
$X_{y}=X\cap\left\{ \left\{ y\right\} \times\mathbb{R}^{m}\right\} $
have dimension at most $k$, the set $Y_{0}\subseteq\mathbb{R}^{n}$
of points $y$ such that the $k$-dimensional volume $v\left(y\right)$
of $X_{y}$ is finite is again globally subanalytic. However, it is
necessary to introduce a function which is not globally subanalytic
in order to express the volume: the restriction of $v$ to $Y_{0}$
has the form $v=P\left(A_{1},\ldots,A_{r},\log A_{1},\ldots,\log A_{r}\right)$,
where $P$ is a polynomial and the $A_{i}$ are positive globally
subanalytic functions.

The class of all functions definable in an o-minimal structure is
closed under many natural operations, but is not in general stable
under parametric integration. For instance, it follows from the above
results that the family $\mathcal{S}$ of all \emph{globally subanalytic
functions }is not stable under parametric integration. However, the
family $\mathcal{C}$ of \emph{constructible functions} (see Definition
\ref{def: subanalytic and constructible}) is, and indeed it is the
smallest such collection containing $\mathcal{S}$ (see \cite{cluckers_miller:loci_integrability}).
Moreover, the locus of integrability of a constructible function is
the zero-set of a function which is again constructible.
 The expansion $\mathbb{R}_{\text{an,\ensuremath{\exp}}}$ of the real field by all restricted analytic functions and the unrestricted exponential
is an o-minimal structure in which all the functions in $\mathcal{C}$ are definable, which is not stable under parametric integration, as shown in \cite[theorem 5.11]{dmm:series}. For instance the error function $\displaystyle x\mapsto \int_{0}^{x}  \mathrm{e}^{-t^2}  \ \mathrm{d}t$ is the parametric integral of a very simple function definable in  $\mathbb{R}_{\text{an,\ensuremath{\exp}}}$, but it is not itself definable in $\mathbb{R}_{\text{an,\ensuremath{\exp}}}$.

Nevertheless, some of these integrals are definable in larger o-minimal structures.
For example, all antiderivatives of functions definable in an o-minimal
structure $\mathcal{R}$ are definable in a larger o-minimal structure,
called the Pfaffian closure of $\mathcal{R}$ \cite{speiss:clos}.
Other parametric integrals and  integral transforms of functions
definable in $\mathbb{R}_{\text{an,\ensuremath{\exp}}}$ (for example,
the restrictions to the real half-line $\left(1,+\infty\right)$ of
the Gamma function, seen as a Mellin transform, and of the Riemann
Zeta function, seen as a quotient of two Mellin transforms) are known
to be definable in suitable larger o-minimal structures \cite{vdd:speiss:multisum,vdd:speiss:gen,rolin_servi_speissegger:multisummability_generalized_series}.
However, there is no known general o-minimal universe in which all
such parametric integrals are definable (and indeed incompatibility
results in \cite{rsw,rss,legal:genericity} suggest that such a universe
might not exist).

We therefore turn our attention  
 to subcollections of functions definable
in a given o-minimal structure (here, 
 $\mathbb{R}_{\text{an,\ensuremath{\exp}}}$)
which are stable under taking parametric integrals, 
as is the family $\mathcal{C}$.
There aren't many
known such collections. For example, the collection of all functions
definable in $\mathbb{R}_{\text{an}}^{\mathrm{pow}}$ (the polynomially bounded expansion of 
$\mathbb{R}_{\text{an}}$
 by all real power functions, seen as
a reduct of $\mathbb{R}_{\text{an},\exp}$) is not stable under parametric
integration and indeed some such integrals are not even definable
in $\mathbb{R}_{\text{an},\exp}$ (see \cite[Prop. 2.1 and Theorem 2.2]{soufflet:asymptotic_expansions} and Subsection \ref{subsec:power-constructible} where this example is discussed in detail).
Our first aim is to define a collection $\mathcal{C}^{\mathbb{R}}$
of $\mathbb{R}$-algebras of functions definable in the o-minimal structure $\mathbb{R}_{\text{an},\exp}$,
extending the stable collection $\mathcal{C}$, and, in turn, stable under parametric integration (see
Definition \ref{def:C^K} and Theorem \ref{thm Stability of C^K}
below, for the case $\mathbb{K}=\mathbb{R}$). The elements of $\mathcal{C}^{\mathbb{R}}$
are called \emph{real} \emph{power-constructible }functions and they
are constructed from real powers and logarithms of globally subanalytic
functions.\medskip{}

Parametric integrals of tame functions also appear in the study of
functional and geometric analogues of \emph{period conjectures}. Recent
breakthroughs in functional transcendence around o-minimality and
periods have been made, concerning the transcendence of the coordinates
of the Hodge filtration, which are ratios of certain period functions.
For instance, Bakker, Klingler and Tsimerman \cite{bakker_klingler_tsimerman:tame_topology_hodge_loci}
proved that period maps are definable in the o-minimal structure $\mathbb{R}_{\mathrm{an,exp}}$,
yielding a new proof of the algebraicity of the Hodge loci. This provides
an example of an integration process whose resulting functions remain
in the original tame framework. Analogously, our Theorem \ref{thm Stability of C^K}
states that parametric integration preserves the class $\mathcal{C}^{\mathbb{R}}$.
In the same spirit, we consider (see Definition \ref{def: C^M} and
Theorems \ref{thm: C_M stability}, \ref{thm:variants}) larger classes
which we prove to be stable under parametric integration.\\

Another motivation for considering the collection $\mathcal{C}^{\mathbb{R}}$
lies beyond o-minimality: most integral transforms (Fourier, Mellin...)
are usually applied to rapidly decaying or compactly supported unary
functions, but they can be extended to classes of functions having
an asymptotic expansion (at $0$ and/or at $\infty$) in the scale
of real power-log monomials (for example, for such functions it can
be shown that the Mellin transform extends to a meromorphic function
on the whole complex plane, outside the domain of convergence of the
integral, see \cite[Section 6.7 (by D. Zagier)]{zagier_the_mellin_transform}.
In order to consider parametric versions of such transforms, one needs
some control over the behaviour of the multi-variable functions in
the collection to which we want to apply the transform. This is clear
for example in the study of oscillatory integrals of the first kind,
when the phase and the amplitude are analytic: resolution of singularities
in the class of analytic germs is used to recover information about
the asymptotic expansion of such parametric integrals. 
This is the strategy developed, for example in \cite{Arnold}, \cite{Malgrange} and \cite{Varchenko}, in which the powers appearing in the asymptotic expansion of certain integral transforms with an analytic phase $f$ (and a compactly supported amplitude) are expressed, using resolution of singularities of $f$, in terms of numerical invariants of the singularity of $f$ at the origin.
When applying
parametric integral transforms to a class $\mathcal{F}$ of functions
in several variables, it is hence important to have information about
the geometry of the domain of the functions in $\mathcal{F}$ and
to have some well-behaved theory of resolution of singularities adapted
to the class $\mathcal{F}$. This is where o-minimality plays a central
role: the key result here is a version of local resolution of singularities
called the \emph{subanalytic preparation theorem} \cite{lr:prep},
\cite{parusinski:lipschitz_stratification_subanalytic}, together
with cell-decomposition and piecewise analyticity arguments to patch
together the local results into a global stability statement.
This viewpoint is implemented in \cite{acrs}, \cite{ccs:FourierMellin}, and in the article in preparation \cite{ccms:rate-decay}, in which we systematically study the rate of decay of Fourier transforms of subanalytic functions, as well as of functions of $\mathcal{C}^\mathbb{R}$. More specifically, we investigate the interplay between rapid decay and holomorphic extension to certain complex domains around the real axis.

Understanding the stability of wide collections of natural functions under oscillatory integral
transforms appears as a key motivation for the theory of distributions. The reader may find in \cite{acrs} an illustration of the importance of controlling asymptotics in the study of certain classes of distributions.

In this paper we study \emph{parametric} Mellin
transforms of functions in $\mathcal{C}^{\mathbb{R}}$, exploiting
both the o-minimal (subanalytic) nature of the domain of the functions
and a preparation theorem available for the functions in $\mathcal{C}^{\mathbb{R}}$.
We define a collection of functions which contains the parametric
Mellin transforms of the functions in $\mathcal{C}^{\mathbb{R}}\left(X\right)$,
for $X\subseteq\mathbb{R}^{m}$ a globally subanalytic set, and stable
under integration with respect to the variable $x\in X$: our starting
point is $\mathcal{C}^{\mathbb{R}},$ a collection of functions defined
on subanalytic sets. We then apply an integral transform which depends
on a complex parameter $s$, which we want to keep separate from the
subanalytic variables, in the sense that we will not integrate with
respect to $s$. For this, we construct a collection $\mathcal{C}^{\mathcal{M}}$
of $\mathbb{C}$-algebras of functions of the variables $\left(s,x\right)$
(where $s$ is a single complex variable and $x$ is a tuple of variables
ranging in a subanalytic set) which contain the parametric Mellin
transforms of power-constructible functions, and stable under parametric
integration. In \cite{ccmrs:integration-oscillatory}, where we considered
the parametric Fourier transforms of constructible functions, the
corresponding system of $\mathbb{C}$-algebras is described in terms
of \emph{transcendental elements}, which are themselves integral operators
evaluated at constructible functions. Here instead we give an explicit
description of parametric Mellin transforms in terms of series of
functions of a simple special form.

The functions in $\mathcal{C}^{\mathcal{M}}$ will be shown to depend
meromorphically on the variable $s$. This, together with Theorem
\ref{thm: C_M stability}, will be used to provide a \emph{meromorphic
extension }of the parametric Mellin transform to the whole complex
plane. A classical result in this spirit is proven in \cite{atiyah:resolution_singularities_division_distribution}
(see also \cite[Th. 1.4]{greenblatt:_resolution_asymptotic_expansion_integrals}
for a more recent and simplified proof): given a real analytic function
$f$ defined in a open neighbourhood $U$of $0\in\mathbb{R}^{n}$,
for every $\mathcal{C}^{\infty}$ function $\varphi$ whose support
is compact and contained in $U$, the integral of $f^{s} \varphi$,
initially defined as a holomorphic function on $\Re\left(s\right)>0$,
extends to a meromorphic function on $\mathbb{C}$.

As the Mellin transform is usually considered as a function of a complex
parameter, we leave the realm of real-valued functions and of o-minimality.
There is hence no reason to restrict ourselves to \emph{real} powers
of subanalytic functions. Therefore, we define \emph{complex power-constructible}
functions, prove that they form a collection $\mathcal{C}^{\mathbb{C}}$
which is stable under parametric integration (see Definition \ref{def:C^K}
and Theorem \ref{thm Stability of C^K} below, case $\mathbb{K}=\mathbb{C}$)
and study their parametric Mellin transforms. The purely imaginary
powers of subanalytic functions introduce now some nontrivial oscillatory
phenomena, which lead us to invoke results from the theory of continuously
distributed functions mod 1 (see Section \ref{subsec:Non-compensation-arguments}).  
Despite the presence of oscillatory functions, which forces us to leave the realm of o-minimality, Theorem \ref{thm: C_M stability} and its consequences show that the class $\mathcal{C}^{\mathcal{M}}$ is geometrically tame, in a broader sense.

\medskip{}

The paper is organized as follows. In Section \ref{sec:Main results},
we introduce several classes of functions, for which we prove stability
under parametric integration: power-constructible functions (Definition
\ref{def:C^K}), parametric power-constructible functions (Definition
\ref{def: C^M}) and some variants (Section \ref{subsubsec:Variants}).
The main results about these classes are stated in Theorems \ref{thm Stability of C^K},
\ref{thm: C_M stability} and \ref{thm:variants}. In Section \ref{sec:Notation-and-general strategy}
we introduce the three basic tools that will be used in the proofs
of the main results: a non-compensation argument about finite sums
of purely imaginary powers, the properties of parametric strong functions
(which are the building blocks in the construction of the class of
parametric power-constructible functions) and the previously mentioned
subanalytic preparation theorem, from which we derive the consequences
needed in our setting. Section \ref{sec:Preparation of powers} is
devoted to preparing the functions in the classes under consideration
in a particularly simple way with respect to a given subanalytic variable.
This will allow in Section \ref{sec:Integration-of-gen of C^M} to
provide a first result about integrating a generator of a class with
respect to a single variable. The proofs of the general stability
statements are carried out in Section \ref{sec:Stability-of (parametric) powers}.

\section{Notation, definitions and main results\label{sec:Main results}}

A subset $X$ of $\mathbb{R}^{m}$ is globally subanalytic if it is
the image under the canonical projection from $\mathbb{R}^{m+n}$
to $\mathbb{R}^{m}$ of a globally semianalytic subset of $\mathbb{R}^{m+n}$
(i.e. a subset $Y\subseteq\mathbb{R}^{m+n}$ such that, in a neighbourhood
of every point of $\mathbb{P}^{1}\left(\mathbb{R}\right)^{m+n}$,
$Y$ is described by finitely many analytic equations and inequalities).
Equivalently, $X$ is definable in the o-minimal structure $\mathbb{R}_{\text{an}}$
(see for example \cite{vdd:d}). Thus, the logarithm $\log:\left(0,+\infty\right)\longrightarrow\mathbb{R}$
and the power map $x^{y}:\left(0,+\infty\right)\times\mathbb{R}\longrightarrow\mathbb{R}$
are functions whose graph is not subanalytic, but they are definable
in the o-minimal structure $\mathbb{R}_{\mathrm{an},\exp}$ (see for
example \cite{dmm:exp}).

Throughout this paper $X\subseteq\mathbb{R}^{m}$ will be a globally
subanalytic set (from now on, just \textquotedblleft \emph{subanalytic}
\emph{set}\textquotedblright , for short). Denote by $\mathcal{S}\left(X\right)$
the collection of all subanalytic functions on $X$, i.e. all the
functions of domain $X$ whose graph is a subanalytic set, and let
$\mathcal{S}_{+}\left(X\right)=\left\{ f\in\mathcal{S}\left(X\right):\ f\left(X\right)\subseteq\left(0,+\infty\right)\right\} $.
\begin{defn}[Constructible functions]
\label{def: subanalytic and constructible}Let $\mathcal{C}\left(X\right)$
be the $\mathbb{R}$-algebra of \emph{constructible} \emph{functions
on $X$,} generated by all subanalytic functions and their logarithms:
\[
\mathcal{C}\left(X\right)=\left\{ \sum_{i=1}^{N}f_{i}\prod_{j=1}^{M}\log g_{i,j}:\ M,N\in\mathbb{N}^{\times},\ f_{i}\in\mathcal{S}\left(X\right),g_{i,j}\in\mathcal{S}_{+}\left(X\right)\right\} .
\]

Define $\mathcal{C}=\left\{ \mathcal{C}\left(X\right):\ \ X\subseteq\mathbb{R}^{m}\text{ subanalytic},\ m\in\mathbb{N}\right\} $.
\end{defn}
By \cite{lr:prep,clr,cluckers-miller:stability-integration-sums-products},
$\mathcal{C}$ is the smallest collection of $\mathbb{R}$-algebras
containing $\mathcal{S}$ and stable under parametric integration.
Notice that constructible functions are definable in $\mathbb{R}_{\mathrm{an,}\exp}$.

A function defined on $X$ and taking its values in $\mathbb{C}$
is called a \emph{complex-valued subanalytic (constructible, resp.)
function} if its real and imaginary parts are in $\mathcal{S}\left(X\right)$
(in $\mathcal{C}\left(X\right)$, resp.).

\subsection{\label{subsec:power-constructible}Power-constructible functions}

For $\mathbb{K}\subseteq\mathbb{C}$ a subfield, write 
\[
\mathcal{S}_{+}^{\mathbb{K}}\left(X\right)=\left\{ f^{\alpha}:\ f\in\mathcal{S}_{+}\left(X\right),\ \alpha\in\mathbb{K}\right\} .
\]

Let $\mathbb{F}_{\mathbb{K}}$ be $\mathbb{R}$ if $\mathbb{K}\subseteq\mathbb{R}$
and $\mathbb{C}$ otherwise.
\begin{defn}[Power-constructible functions]
\label{def:C^K}Let $\mathcal{C}^{\mathbb{K}}\left(X\right)$ be
the $\mathbb{F}_{\mathbb{K}}$-algebra generated by the logarithms
and the $\mathbb{K}$-powers of the subanalytic functions on $X$,
i.e.
\[
\mathcal{C}^{\mathbb{K}}\left(X\right)=\left\{ \sum_{i=1}^{N}c_{i}\prod_{j=1}^{M}f_{i,j}^{\alpha_{i,j}}\log g_{i,j}:\ M,N\in\mathbb{N}^{\times},\ f_{i,j},g_{i,j}\in\mathcal{S}_{+}\left(X\right),\ \alpha_{i,j}\in\mathbb{K},c_{i}\in\mathbb{F}_{\mathbb{K}}\right\} .
\]

Let
\[
\mathcal{C}^{\mathbb{\mathbb{K}}}=\left\{ \mathcal{C}^{\mathbb{\mathbb{K}}}\left(X\right):\ X\subseteq\mathbb{R}^{m}\text{ subanalytic},\ m\in\mathbb{N}\right\} .
\]
The functions in $\mathcal{C}^{\mathbb{K}}$ are called \emph{$\mathbb{K}$-power-constructible
functions}.
\end{defn}
\begin{rem}
\label{rem: real and im part of power-constr}Notice that $\mathcal{C}^{\mathbb{Q}}=\mathcal{C}$
and if $\mathbb{K}\subseteq\mathbb{R}$, then the functions in $\mathcal{C}^{\mathbb{K}}$
are definable in $\mathbb{R}_{\mathrm{an,}\exp}$. If $\mathbb{K}\not\subseteq\mathbb{R}$,
then by definition $\mathcal{C}^{\mathbb{K}}\left(X\right)$ is a
$\mathbb{C}$-algebra. However, if $h=\sum c_{i}\prod f_{i,j}^{\alpha_{i,j}}\log g_{i,j}$
is such that all the exponents $\alpha_{i,j}$ belong to $\mathbb{R}$,
then the real and imaginary parts of $h$ belong to the $\mathbb{R}$-algebra
$\mathcal{C}^{\mathbb{R}}\left(X\right)$. 
\end{rem}
Let $\mathbb{R}_{\text{alg}}$ be the field of real algebraic numbers
and consider the expansion $\mathbb{R}_{\text{an}}^{\mathbb{R}_{\text{alg}}}$
of $\mathbb{R}_{\text{an }}$ by all power functions with exponents
in $\mathbb{R}_{\text{alg}}$. It is shown in \cite{kaiser:first_order_tameness_measures}
that the parametric integrals of all the functions definable in $\mathbb{R}_{\text{an}}^{\mathbb{R}_{\text{alg}}}$
are definable in $\mathbb{R}_{\text{an},\exp}$. However, this is
not the case if we allow the exponents of the power functions to range
in the whole field $\mathbb{R}$. 
Indeed, in \cite[Prop. 2.1]{soufflet:asymptotic_expansions}
the author produces an example of a function $f$ in two variables
$x$ and $y$, defined as a composition of subanalytic functions and
irrational powers (in particular, definable in the o-minimal structure $\mathbb{R}_{\mathrm{an,}\exp}$
and even in $\mathbb{R}_{\text{an}}^{\mathrm{pow}}$), such that the
parametric integral (with respect to $y$) of $f$ is not definable
in $\mathbb{R}_{\mathrm{an,}\exp}$. 
The argument goes as follows: Soufflet proves that functions definable in $\mathbb{R}_{\text{an},\exp}$ that have a formal asymptotic expansion in a logarithmic scale (the real scale $\mathfrak{E}_\mathbf{R}$ defined in \cite[p. 129]{soufflet:asymptotic_expansions}) have the property that such an expansion is convergent (see  \cite[Theorem 2.5]{soufflet:asymptotic_expansions}). Now, in \cite[Proposition 2.1]{soufflet:asymptotic_expansions} he shows that the parametric integral of $f$ has a divergent asymptotic expansion in this scale.
More precisely the function $f$
is obtained by right-composing a subanalytic function by a suitable
irrational power of the variable $y$. This procedure differs from
the one in the above definition, where we left-compose subanalytic
functions with irrational powers. Indeed, $f$ is not power-constructible,
as our first result (Theorem \ref{thm Stability of C^K} below) is
that $\mathcal{C}^{\mathbb{K}}$ is stable under parametric integration.
\begin{thm}
\label{thm Stability of C^K}Let $h\in\mathcal{C}^{\mathbb{K}}\left(X\times\mathbb{R}^{n}\right)$.
There exists $H\in\mathcal{C}^{\mathbb{K}}\left(X\right)$ such that
\[
\forall x\in\mathrm{Int}\left(h;X\right),\ \int_{\mathbb{R}^{n}}h\left(x,y\right)\text{d}y=H\left(x\right),
\]
where 
\[
\mathrm{Int}\left(h;X\right):=\left\{ x\in X:\ y\longmapsto h\left(x,y\right)\in L^{1}\left(\mathbb{R}^{n}\right)\right\} .
\]
\end{thm}

\subsection{\label{subsec:strong functions}Strong functions}

In the subanalytic and constructible preparation theorems, a special
role is played by the so-called \emph{strong functions}: these are
bounded subanalytic functions which can be expressed as the composition
of a single power series (convergent in a neighbourhood of the closed
unit polydisk) with a bounded subanalytic map. In order to define
parametric Mellin transforms, we will need a parametric version of
strong functions, where the parameter will be the complex number $s$
appearing in the integration kernel of the Mellin transform.

We first give the definition of a subanalytic strong function and
then proceed to define its parametric counterpart.
\begin{defn}
\label{def: strongly bdd}For $N\in\mathbb{N}$, we let $\mathcal{S}_{c}^{N}\left(X\right)$
be the collection of all maps $\psi:X\longrightarrow\mathbb{R}^{N}$
with components in $\mathcal{S}\left(X\right)$, such that $\overline{\psi\left(X\right)}$
is contained in the closed polydisk of $\mathbb{R}^{N}$ centreed
at zero and of radius $1$. We call
\[
\mathcal{S}_{c}\left(X\right)=\bigcup_{N\in\mathbb{N^{\times}}}\mathcal{S}_{c}^{N}\left(X\right)
\]
the collection of all \emph{1-bounded} subanalytic maps defined on
$X$.
\end{defn}
The following definition is inspired by \cite[Definition 3.3]{cluckers_miller:loci_integrability}
and \cite[Definition 3.6]{ccmrs:integration-oscillatory}.
\begin{defn}[Strong functions]
\label{def: classical strong}We say that $W:X\longrightarrow\mathbb{F}_{\mathbb{K}}$
is an \emph{$\mathbb{F}_{\mathbb{K}}$-valued} \emph{subanalytic strong
function} if there are $N\in\mathbb{N}^{\times}$, a $1$-bounded
subanalytic map $\psi:X\longrightarrow\mathbb{R}^{N}$ and a series
$F\in\mathbb{F}_{\mathbb{K}}\left\llbracket Z\right\rrbracket $ in
$N$ variables $Z$, which converges in a neighbourhood of the closed
polydisk $D^{N}$ centreed at zero and of radius $\frac{3}{2}$ in
$\mathbb{R}^{N}$ (we will say for short that $F$ converges \emph{strongly},
see below), such that $W=F\circ\psi$. If furthermore $\left|F-1\right|<\frac{1}{2}$,
the function $W$ is called a \emph{strong unit} (see \cite[Remarks 3.7]{ccmrs:integration-oscillatory}). 
\end{defn}
We are now ready to define parametric strong functions: these can
be written as certain convergent series composed with $1$-bounded
subanalytic maps, but the coefficients of the series are now (meromorphic)
functions of a complex parameter $s$.
\begin{defn}
\label{def: strong convergence}Let $\mathcal{E}$ be the field of
meromorphic functions $\xi:\mathbb{C}\longrightarrow\mathbb{C}$ and
denote by $D^{N}$ the closed polydisk of radius $\frac{3}{2}$ and
centre $0\in\mathbb{R}^{N}$.

Given a formal power series $F=\sum_{I}\xi_{I}\left(s\right)Z^{I}\in\mathcal{E}\left\llbracket Z\right\rrbracket $
in $N$ variables $Z$ and with coefficients $\xi_{I}\in\mathcal{E}$,
we say that $F$ \emph{converges strongly} if there exists a closed
discrete set $P\left(F\right)\subseteq\mathbb{C}$ (called the \emph{set
of poles} of $F$) such that:
\begin{itemize}
\item for every $s_{0}\in\mathbb{C}\setminus P\left(F\right)$, the power
series $F\left(s_{0},Z\right)\in\mathbb{C}\left\llbracket Z\right\rrbracket $
converges in a neighbourhood of $D^{N}$ (thus $F$ defines a function
on $\left(\mathbb{C}\setminus P\left(F\right)\right)\times D^{N}$);
\item for every $s_{0}\in\mathbb{C}$ there exists $m=m\left(s_{0}\right)\in\mathbb{N}$
such that for all $z_{0}\in D^{N}$, the function $\left(s,z\right)\longmapsto\left(s-s_{0}\right)^{m}F\left(s,z\right)$
has a holomorphic extension on some complex neighbourhood of $\left(s_{0},z_{0}\right)$
\item $P\left(F\right)$ is the set of all $s_{0}\in\mathbb{C}$ such that
the minimal such $m\left(s_{0}\right)$ is strictly positive. 
\end{itemize}
\end{defn}
\begin{rem}
\label{rem: strong convergence}It is easy to see that $P\left(F\right)$
coincides with the set of poles of the coefficients $\xi_{I}$ and
that for each $s_{0}\in P\left(F\right)$ there is an integer $m\in\mathbb{N}$
such that for all $I,\ \text{ord}_{s_{0}}\left(\xi_{I}\right)\leq m$.
\end{rem}
\begin{defn}[Parametric strong functions]
\label{def: parametric strong}Given a closed discrete set $P\subseteq\mathbb{C}$,
a function $\Phi:\left(\mathbb{C}\setminus P\right)\times X\longrightarrow\mathbb{C}$
is called a \emph{parametric strong function} on $X$ if there exist
a $1$-bounded subanalytic map $\psi\in\mathcal{S}_{c}^{N}\left(X\right)$
and a strongly convergent series $F=\sum_{I}\xi_{I}\left(s\right)Z^{I}\in\mathcal{E}\left\llbracket Z\right\rrbracket $
with $P\left(F\right)\subseteq P$ such that,
\[
\forall\left(s,x\right)\in\left(\mathbb{C}\setminus P\right)\times X,\ \Phi\left(s,x\right)=F\circ\left(s,\psi\left(x\right)\right)=\sum_{I}\xi_{I}\left(s\right)\left(\psi\left(x\right)\right)^{I}.
\]
Define $\mathcal{A}\left(X\right)$ as the collection of all parametric
strong functions on $X$ (defined on sets of the form $\left(\mathbb{C}\setminus P\right)\times X$,
for any closed discrete $P\subseteq\mathbb{C}$). Note that if $X\subseteq\mathbb{R}^{0}$
then $\mathcal{A}\left(X\right)=\mathcal{E}$. We let
\[
\mathcal{A}=\left\{ \mathcal{A}\left(X\right):\ X\subseteq\mathbb{R}^{m}\text{ subanalytic},\ m\in\mathbb{N}\right\} .
\]
\end{defn}
\begin{rem}
\label{rem: defined outside P}Since the same $\Phi\in\mathcal{A}\left(X\right)$
could be presented by two different series $F$ with different poles,
we will say ``let $\Phi\in\mathcal{A}\left(X\right)$ have no poles
outside some closed discrete set $P\subseteq\mathbb{C}$'' to mean
that there exist $F,\psi$ such that $\Phi=F\circ\left(s,\psi\right)$
and $P\left(F\right)\subseteq P$. By the same argument, $\mathcal{A}\left(X\right)$
is a $\mathbb{C}$-algebra, up to defining the sum and product on
a common domain. 
\end{rem}

\subsection{\label{subsec:Parametric powers}Parametric powers and the Mellin
transform}

We introduce two parametric integral operator which will be the object
of our study.
\begin{defn}
\label{def: param power and Mellin tr}$\ $
\begin{itemize}
\item For $X\subseteq\mathbb{R}^{m}$ subanalytic, define
\[
\mathcal{P}\left(\mathcal{S}_{+}\left(X\right)\right)=\{P_{f}:\mathbb{C}\times X\longrightarrow\mathbb{C}\text{\ such\ that}\ P_{f}\left(s,x\right)=f\left(x\right)^{s},\ \text{for\ some\ }f\in\mathcal{S}_{+}\left(X\right)\}.
\]
 The \emph{parametric powers of $\mathcal{S}$} are the functions
in the collection
\begin{align*}
\mathcal{P}\left(\mathcal{S}_{+}\right) & =\{\mathcal{P}\left(\mathcal{S}_{+}\left(X\right)\right):\ X\subseteq\mathbb{R}^{m}\text{ subanalytic},\ m\in\mathbb{N}\}.
\end{align*}
\item Let $\mathcal{F}=\left\{ \mathcal{F}\left(X\right):\ X\subseteq\mathbb{R}^{m}\text{ subanalytic},\ m\in\mathbb{N}\right\} $
be a collection of real- or complex-valued functions and $\Sigma\subseteq\mathbb{C}$.
If $f\in\mathcal{F}\left(X\times\mathbb{R}\right)$ is such that for
all $\left(s,x\right)\in\Sigma\times X,\ y\longmapsto y^{s-1}f\left(x,y\right)\in L^{1}\left(\mathbb{R}^{>0}\right)$,
define the parametric Mellin transform of $f$ on $\Sigma\times X$
as the function 
\[
\mathcal{M}_{\Sigma}\left[f\right]\left(s,x\right)=\int_{0}^{+\infty}y^{s-1}f\left(x,y\right)\text{d}y,\ \forall\left(s,x\right)\in\Sigma\times X.
\]
\\
The \emph{parametric Mellin transforms of $\mathcal{F}$ on $\Sigma$}
are the elements of the collection 
\[
\mathcal{M}_{\Sigma}\left[\mathcal{F}\right]=\left\{ \mathcal{M}_{\Sigma}\left[f\right]:\ f\text{ as above, for some }X\right\} .
\]
\end{itemize}
\end{defn}
Our next aim is to define a collection of algebras of functions which
is stable under parametric integration and which contains both the
parametric powers of $\mathcal{S}$ and the Mellin transforms of $\mathcal{C}^{\mathbb{C}}$
on $\mathbb{C}$ (Definition \ref{def: C^M}). In order to motivate
the definition, let us give three simple examples.
\begin{examples}
\label{exs: param powers and Mellin}Let $X\subseteq\mathbb{R}^{m}$
be subanalytic and $a,b\in\mathcal{S}\left(X\right)$ be such that
for all $x\in X,\ 1\leq a\left(x\right)\leq2\leq b\left(x\right)$.
\begin{enumerate}[resume]
\item Let $\chi_{1}\left(x,y\right)$ be the characteristic function of
the set
\[
B_{1}=\left\{ \left(x,y\right):\ x\in X,\ 0<y<a\left(x\right)\right\} 
\]
 and consider the subanalytic function
\[
f\left(x,y\right)=\chi_{1}\left(x,y\right)\frac{a\left(x\right)b\left(x\right)}{a\left(x\right)b\left(x\right)-y}\in\mathcal{S}\left(X\times\mathbb{R}\right).
\]
Since $1\leq f\left(x,y\right)\leq2$, the parametric Mellin transform
of $f$ is well defined on $\Sigma_{1}=\left\{ s\in\mathbb{C}:\ \Re\left(s\right)>0\right\} $,
is holomorphic in $s$ and is given by
\begin{align*}
\mathcal{M}_{\Sigma_{1}}\left[f\right]\left(s,x\right) & =\int_{0}^{a\left(x\right)}y^{s-1}\frac{a\left(x\right)b\left(x\right)}{a\left(x\right)b\left(x\right)-y}\text{d}y\\
 & =\int_{0}^{a\left(x\right)}y^{s-1}\sum_{k\geq0}\left(\frac{y}{a\left(x\right)b\left(x\right)}\right)^{k}\text{d}y.
\end{align*}
The series in the above integral converges normally on $B_{1}$, hence
we can permute sum and integral and write
\begin{align*}
\mathcal{M}_{\Sigma_{1}}\left[f\right]\left(s,x\right) & =\sum_{k\geq0}\left(a\left(x\right)b\left(x\right)\right)^{-k}\int_{0}^{a\left(x\right)}y^{s-1+k}\text{d}y\\
 & =\left(a\left(x\right)\right)^{s}\sum_{k\geq0}\frac{\left(b\left(x\right)\right)^{-k}}{s+k}.
\end{align*}
Notice that in this computation we create both the parametric power
of a subanalytic function, and a series of functions depending on
the complex variable $s$ and on the real variable $x$. The above
series defines a parametric strong function on $\mathbb{C}\times X$,
with poles at zero and at the negative integers.
\item Let $\Sigma_{2}=\left\{ s\in\mathbb{C}:\ \Re\left(s\right)<1\right\} $
and $\chi_{2}\left(x,y\right)$ be the characteristic function of
the set
\[
B_{2}=\left\{ \left(x,y\right):\ x\in X,\ y>a\left(x\right)\right\} .
\]
Consider the subanalytic function $g\left(x,y\right)=\chi_{2}\left(x,y\right)y\left(1+\frac{a\left(x\right)}{b\left(x\right)y}\right)\in\mathcal{S}\left(X\times\mathbb{R}\right)$.
We aim to compute the parametric integral (with respect to the variable
$y$) of the function $y^{-2}\left(g\left(x,y\right)\right)^{s}$.
Since $0\leq\frac{a\left(x\right)}{b\left(x\right)y}\leq\frac{1}{2}$
on $B_{2}$, such an integral exists on $\Sigma_{2}\times X$ and
\begin{align*}
I\left(g;\Sigma_{2}\times X\right):=\int_{a\left(x\right)}^{+\infty}y^{s-2}\left(1+\frac{a\left(x\right)}{b\left(x\right)y}\right)^{s}\text{d}y & =\int_{a\left(x\right)}^{+\infty}y^{s-2}\sum_{k}{s \choose k}\left(\frac{a\left(x\right)}{b\left(x\right)y}\right)^{k}\text{d}y\\
=\sum_{k}{s \choose k}\left(\frac{a\left(x\right)}{b\left(x\right)}\right)^{k}\int_{a\left(x\right)}^{+\infty}y^{s-2-k}\text{d}y & =-\left(a\left(x\right)\right)^{s-1}\sum_{k}{s \choose k}\frac{\left(b\left(x\right)\right)^{-k}}{s-1-k}.
\end{align*}
Again, the above series defines a parametric strong function on $\mathbb{C}\times X$,
with poles at the positive integers.
\item Let $h\left(s,x,y\right)=f\left(x,y\right)y^{s-1}+y^{-2}\left(g\left(x,y\right)\right)^{s}$.
Direct calculation shows that, letting
\[
\text{Int}\left(h;\left(\mathbb{C}\setminus\mathbb{Z}\right)\times X\right):=\left\{ \left(s,x\right)\in\left(\mathbb{C}\setminus\mathbb{Z}\right)\times X:\ y\longmapsto h\left(s,x,y\right)\in L^{1}\left(\mathbb{R}\right)\right\} ,
\]
we have 
\[
\text{Int}\left(h;\left(\mathbb{C}\setminus\mathbb{Z}\right)\times X\right)=\left\{ s\in\mathbb{C}\setminus\mathbb{Z}:\ 0<\Re\left(s\right)<1\right\} \times X.
\]
However, the function $H$ defined on $\left(\mathbb{C}\setminus\mathbb{Z}\right)\times X$
by
\[
H\left(s,x\right)=\left(a\left(x\right)\right)^{s}\sum_{k\geq0}\frac{\left(b\left(x\right)\right)^{-k}}{s+k}-\left(a\left(x\right)\right)^{s-1}\sum_{k}{s \choose k}\frac{\left(b\left(x\right)\right)^{-k}}{s-1-k}
\]
depends meromorphically on $s$ and can be seen as an interpolation
of the integral of $h$ on the whole $\mathbb{C}\times X$.
\end{enumerate}
Given a subanalytic set $X\subseteq\mathbb{R}^{m}$, recall the definitions
of the algebras $\mathcal{A}\left(X\right)$ of parametric strong
functions and $\mathcal{C}^{\mathbb{C}}\left(X\right)$ of $\mathbb{C}$-power-constructible
functions, and of the collection $\mathcal{P}\left(\mathcal{S}_{+}\left(X\right)\right)$
of parametric powers of subanalytic functions (Definitions \ref{def: parametric strong},
\ref{def:C^K} and \ref{def: param power and Mellin tr}).
\end{examples}
\begin{defn}
\label{def: C^M}If $X\subseteq\mathbb{R}^{0}$, then define $\mathcal{C}^{\mathcal{M}}\left(X\right)=\mathcal{E}$.
If $X\subseteq\mathbb{R}^{m}$, with $m>0$, then we let $\mathcal{C}^{\mathcal{M}}\left(X\right)$
be the $\mathcal{A}\left(X\right)$-algebra generated by $\mathcal{C}^{\mathbb{C}}\left(X\right)\cup\mathcal{P}\left(\mathcal{S}_{+}\left(X\right)\right)$.
Every function $h\in\mathcal{C}^{\mathcal{M}}\left(X\right)$ can
be written on $\left(\mathbb{C}\setminus P\right)\times X$ (for some
closed discrete $P\subseteq\mathbb{C}$) as a finite sum of \emph{generators}
of the form 
\begin{equation}
\Phi\left(s,x\right)\cdot g\left(x\right)\cdot f\left(x\right)^{s},\label{eq:generator}
\end{equation}
where $g\in\mathcal{C}^{\mathbb{C}}\left(X\right),\ f\in\mathcal{S}_{+}\left(X\right)$
and $\Phi\in\mathcal{A}\left(X\right)$ has no poles outside $P$.

If $h\in\mathcal{C}^{\mathcal{M}}\left(X\right)$ can be presented
as a sum of generators in which the parametric strong functions have
no poles outside some common set $P\subseteq\mathbb{C}$, then we
say that $h$ has no poles outside $P$. We let 
\[
\mathcal{C}^{\mathcal{M}}=\left\{ \mathcal{C}^{\mathcal{M}}\left(X\right):\ X\subseteq\mathbb{R}^{m}\text{ subanalytic},\ m\in\mathbb{N}\right\} 
\]
be the collection of algebras of (complex) \emph{parametric power-constructible
functions}.
\end{defn}
Remark \ref{rem: defined outside P} also applies to the functions
in $\mathcal{C}^{\mathcal{M}}\left(X\right)$. 
\begin{rem}
If $h\in\mathcal{C}^{\mathcal{M}}\left(X\right)$ has no poles outside
some closed discrete set $P,$ then for all $s\in\mathbb{C}\setminus P,\ x\longmapsto h\left(s,x\right)\in\mathcal{C}^{\mathbb{C}}\left(X\right)$
and the dependence on the variables $x$ is piecewise analytic, by
o-minimality. Moreover, by definition of $\mathcal{A}\left(X\right)$,
for all $x\in X,\ s\longmapsto h\left(s,x\right)$ is meromorphic
on $\mathbb{C}$.
\end{rem}
The main goal of this paper is to study the nature of the parametric
integrals of functions in $\mathcal{C}^{\mathcal{M}}$. Let $X\subseteq\mathbb{R}^{m}$
be subanalytic, and consider a function $h\in\mathcal{C}^{\mathcal{M}}\left(X\times\mathbb{R}^{n}\right)$
without poles outside some closed and discrete set $P\subseteq\mathbb{C}$.
Then $h$ depends on a complex variable $s$ and on $\left(m+n\right)$
real variables (let us call them $x$, ranging in $X$, and $y$,
ranging in $\mathbb{R}^{n}$). We integrate $h$ in the variables
$y$ over $\mathbb{R}^{n}$, whenever the integral exists, and study
the nature of the resulting function. 

The set of parameters $\left(s,x\right)\in\left(\mathbb{C}\setminus P\right)\times X$
for which the integral exists is the \emph{integration locus} of $h$. 
\begin{defn}
\label{Def:IntMellin}For $h\in\mathcal{C}^{\mathcal{M}}\left(X\times\mathbb{R}^{n}\right)$
and a closed discrete set $P\subseteq\mathbb{C}$ such that $h$ has
no poles outside $P$, define
\[
\mathrm{Int}\left(h;\left(\mathbb{C}\setminus P\right)\times X\right):=\left\{ \left(s,x\right)\in\left(\mathbb{C}\setminus P\right)\times X:\ y\longmapsto h\left(s,x,y\right)\in L^{1}\left(\mathbb{R}^{n}\right)\right\} .
\]

Our main result is the following.
\end{defn}
\begin{thm}
\label{thm: C_M stability}Let $h\in\mathcal{C}^{\mathcal{M}}\left(X\times\mathbb{R}^{n}\right)$
be without poles outside some closed discrete set $P\subseteq\mathbb{C}$.
There exist a closed discrete set $P'\subseteq\mathbb{C}$ containing
$P$ and $H\in\mathcal{C}^{\mathcal{M}}\left(X\right)$ with no poles
outside $P'$ such that
\[
\forall\left(s,x\right)\in\mathrm{Int}\left(h;\left(\mathbb{C}\setminus P'\right)\times X\right),\ H\left(s,x\right)=\int_{\mathbb{R}^{n}}h\left(s,x,y\right)\mathrm{d}y.
\]
Moreover, $P'\setminus P$ is contained in a finitely generated $\mathbb{Z}$-lattice.
\end{thm}
The above examples and Theorem \ref{thm: C_M stability} suggest to
introduce the following definition.
\begin{defn}
\label{def: gen param mellin transf}Let $\mathcal{G}\left(X\right)$
be a collection of functions $f:\mathbb{C}\times X\longrightarrow\mathbb{C}$,
where $X\subseteq\mathbb{R}^{m}$ is subanalytic and $f$ depends
meromorphically on its complex variable $s$. We say that $\mathcal{G}=\left\{ \mathcal{G}\left(X\right):\ X\subseteq\mathbb{R}^{m},\ m\in\mathbb{N}\right\} $
is \emph{stable under generalized parametric Mellin transform} if
whenever $f\in\mathcal{G}$$\left(X\times\mathbb{R}\right)$ has no
poles outside some closed discrete set $P\subseteq\mathbb{C}$, there
exist a closed discrete set $P'\subseteq\mathbb{C}$ containing $P$
and $\mathcal{M}_{f}\in\mathcal{G}\left(X\right)$ without poles outside
$P'$ such that, if $g\left(s,x,y\right)=y^{s-1}f\left(s,x,y\right)\chi_{\left(0,+\infty\right)}\left(y\right)$,
then 
\[
\forall\left(s,x\right)\in\mathrm{Int}\left(g;\left(\mathbb{C}\setminus P'\right)\times X\right),\ \mathcal{M}_{f}\left(s,x\right)=\int_{0}^{+\infty}y^{s-1}f\left(s,x,y\right)\text{d}y.
\]
\end{defn}
\begin{cor}
\label{cor: C^M smallest system}$\mathcal{C}^{\mathcal{M}}$ is the
smallest system of $\mathcal{A}$-algebras containing $\mathcal{\mathcal{C}}^{\mathbb{C}}$
and stable under the generalized parametric Mellin transform.
\end{cor}
\begin{proof}
By Theorem \ref{thm: C_M stability}, $\mathcal{C}^{\mathcal{M}}$
is such a system. Let us show that it is the smallest.

Let $f\in\mathcal{S}\left(X\right)$. Let $y$ be a single variable
and let $\chi\left(x,y\right)$ be the characteristic function of
the set $\left\{ \left(x,y\right):\ 0<y<\left|f\left(x\right)\right|\right\} $
and consider the parametric Mellin transform of the function $\left(s,x,y\right)\longmapsto f\left(s,x,y\right)=s\cdot\chi\left(x,y\right)$
on $\Sigma=\left\{ s\in\mathbb{C}:\ \Re\left(s\right)>0\right\} $:
\begin{align*}
\mathcal{M}_{\Sigma}\left[f\right]\left(s,x\right) & =\int_{0}^{+\infty}sy^{s-1}\chi\left(x,y\right)\text{d}y\\
 & =s\int_{0}^{\left|f\left(x\right)\right|}y^{s-1}\text{d}y=\left|f\left(x\right)\right|^{s}.
\end{align*}
If $\mathcal{D}$ is a system of $\mathcal{A}$-algebras containing
$\mathcal{\mathcal{C}}^{\mathbb{C}}$, then $\mathcal{D}$ contains
the function $f$, and if $\mathcal{D}$ is stable under the generalized
parametric Mellin transform, then $\mathcal{D}$ contains the extension
$\mathcal{M}_{f}$ of $\mathcal{M}_{\Sigma}\left[f\right]$ to the
whole complex plane. Hence $\mathcal{P}\left(\mathcal{S}_{+}\right)\subseteq\mathcal{D}$,
i.e. $\mathcal{C}^{\mathcal{M}}\subseteq\mathcal{D}$.
\end{proof}

\subsubsection{\label{subsubsec:Variants}Parametric powers of $\mathbb{K}$-power-subanalytic
functions}

We consider several collections, defined via minor variations of the
definition of $\mathcal{C}^{\mathcal{M}}$, and which we will prove
to be stable under parametric integration.

Let $\mathbb{K}\subseteq\mathbb{C}$ be a subfield.

In Definition \ref{def: parametric strong}, we replace $\mathcal{E}$
by
\[
\mathcal{E}_{\mathbb{K}}:=\left\{ \xi\in\mathcal{E}:\ P\left(\xi\right)\subseteq\mathbb{K}\right\} 
\]
(where $P\left(\xi\right)$ is the set of poles of $\xi$) and we
define $\mathcal{A}_{\mathbb{K}}$ accordingly.

We let $\mathcal{C}^{\mathbb{K},\mathcal{M}}\left(X\right)$ be the
$\mathcal{A}_{\mathbb{K}}\left(X\right)$-algebra generated by $\mathcal{C}^{\mathbb{K}}\left(X\right)\cup\mathcal{P}\left(\mathcal{S}_{+}\left(X\right)\right)$.
Every element of $\mathcal{C}^{\mathbb{K},\mathcal{M}}\left(X\right)$
can be written as a finite sum of generators of the form \eqref{eq:generator},
where now $\Phi\in\mathcal{A}_{\mathbb{K}}\left(X\right)$ and $g\in\mathcal{C}^{\mathbb{K}}\left(X\right)$.

\medskip{}

Next, we define a similar system of algebras which furthermore contains
the parametric powers of $\mathbb{K}$-powers of subanalytic functions.
For this, given $X\subseteq\mathbb{R}^{m}$ subanalytic, let
\[
\mathcal{P}\left(\mathcal{S}_{+}^{\mathbb{K}}\left(X\right)\right)=\left\{ P_{f}:\mathbb{C}\times X\longrightarrow\mathbb{C}\text{\ such\ that}\ P_{f}\left(s,x\right)=f\left(x\right)^{\alpha s},\ \text{for\ some\ }f\in\mathcal{S}_{+}\left(X\right)\text{ and }\alpha\in\mathbb{K}\right\} 
\]
and $\mathcal{C}^{\mathcal{P}\left(\mathbb{K}\right),\mathcal{M}}\left(X\right)$
be the $\mathcal{A}_{\mathbb{K}}\left(X\right)$-algebra generated
by $\mathcal{C}^{\mathbb{K}}\left(X\right)\cup\mathcal{P}\left(\mathcal{S}_{+}^{\mathbb{K}}\left(X\right)\right)$.
Every element of $\mathcal{C}^{\mathcal{P}\left(\mathbb{K}\right),\mathcal{M}}\left(X\right)$
can be written as a finite sum of generators of the form
\[
\Phi\left(s,x\right)g\left(x\right)f_{1}\left(x\right)^{\alpha_{1}s}\cdots f_{n}\left(x\right)^{\alpha_{n}s},
\]
where $\Phi\in\mathcal{A}_{\mathbb{K}}\left(X\right),\ g\in\mathcal{C}^{\mathbb{K}}\left(X\right),\ n\in\mathbb{N},\ \alpha_{i}\in\mathbb{K}$
and $f_{i}\in\mathcal{S}_{+}\left(X\right)$.

Let 
\begin{align*}
\mathcal{C}^{\mathbb{K},\mathcal{M}} & =\left\{ \mathcal{C}^{\mathbb{K},\mathcal{M}}\left(X\right):\ X\subseteq\mathbb{R}^{m}\text{ subanalytic},\ m\in\mathbb{N}\right\} ,\\
\mathcal{C}^{\mathcal{P}\left(\mathbb{K}\right),\mathcal{M}} & =\left\{ \mathcal{C}^{\mathcal{P}\left(\mathbb{K}\right),\mathcal{M}}\left(X\right):\ X\subseteq\mathbb{R}^{m}\text{ subanalytic},\ m\in\mathbb{N}\right\} .
\end{align*}

\begin{thm}
\label{thm:variants}The statement of Theorem \ref{thm: C_M stability}
also holds if we replace $\mathcal{C}^{\mathcal{M}}$ by either $\mathcal{C}^{\mathbb{K},\mathcal{M}}$
or $\mathcal{C}^{\mathcal{P}\left(\mathbb{K}\right),\mathcal{M}}$
(the closed discrete set $P'$ is now contained in $\mathbb{K}$).
\end{thm}
Arguing as in Corollary \ref{cor: C^M smallest system}, it follows
that $\mathcal{C}^{\mathbb{K},\mathcal{M}}$ is the smallest system
of $\mathcal{A}_{\mathbb{K}}$-algebras containing $\mathcal{C}^{\mathbb{K}}$
and stable under parametric Mellin transform. Notice that the collection
$\mathcal{C}^{\mathbb{K}}$ of $\mathbb{K}$-power-constructible functions
coincides with the collection of all functions in $\mathcal{C}^{\mathbb{K},\mathcal{M}}$
which happen not to depend on the parameter $s$.

The system of $\mathcal{A}_{\mathbb{K}}$-algebras $\mathcal{C}^{\mathcal{P}\left(\mathbb{K}\right),\mathcal{M}}$
also contains $\mathcal{C}^{\mathbb{K}}$ and is stable under parametric
Mellin transform. As a consequence of the proof of Theorem \ref{thm:variants},
we show (see Theorem \ref{thm: interpolation and locus C^M} and Remark
\ref{rem: non vertical grid}) that the system $\mathcal{C}^{\mathcal{P}\left(\mathbb{C}\right),\mathcal{M}}$
is strictly larger than $\mathcal{C}^{\mathcal{M}}$.

\section{Toolbox\label{sec:Notation-and-general strategy}}

Throughout this paper, $X\subseteq\mathbb{R}^{m}$ is a subanalytic
set which serves as space of parameters (we never integrate with respect
to the variables ranging in $X$). Since all the classes $\mathcal{D}$
of functions defined in Sections \ref{subsec:power-constructible}
and \ref{subsec:Parametric powers} are stable under multiplication
by a subanalytic function, when studying a function $f\in\mathcal{D}\left(X\times\mathbb{R}^{n}\right)$,
we are allowed to partition $X$ into subanalytic cells, replace $X$
by one of the cells of the partition and work disjointly in restriction
to such a cell. In particular, we may always assume that $X$ is itself
a subanalytic cell, and that all cells in $X\times\mathbb{R}^{n}$
project onto $X$.

If $\mathcal{D}$ is any of the classes defined in Sections \ref{subsec:power-constructible}
and $f\in\mathcal{D}\left(X\times\mathbb{R}^{n}\right)$, we often
compute the integral of $f$ with respect to the variables ranging
in $\mathbb{R}^{n}$. If $X\times\mathbb{R}^{n}$ is partitioned into
finitely many subanalytic cells, then only the cells which have nonempty
interior in $X\times\mathbb{R}^{n}$ contribute to the integral. This
motivates the following definition.
\begin{defn}
\label{def:cellopen over X}Let $A\subseteq X\times\mathbb{R}$ be
a subanalytic cell. We say that $A$ is \emph{open over $X$ }if there
are $\varphi_{1},\varphi_{2}\in\mathcal{S}\left(X\right)\cup\left\{ \pm\infty\right\} $
such that for all $x\in X,\ \varphi_{1}\left(x\right)<\varphi_{2}\left(x\right)$
and
\[
A=\left\{ \left(x,y\right):\ x\in X,\ \varphi_{1}\left(x\right)<y<\varphi_{2}\left(x\right)\right\} .
\]
\end{defn}
\begin{notation}
For $x\in X$, define $A_{x}=\left\{ y\in\mathbb{R}^{n}:\ \left(x,y\right)\in A\right\} $. 

Hence, if $f\in\mathcal{D}\left(A\right)$, then 
\[
\text{Int}\left(f;\mathbb{C}\times X\right)=\left\{ \left(s,x\right)\in\mathbb{C}\times X:\ y\longmapsto f\left(s,x,y\right)\in L^{1}\left(A_{x}\right)\right\} .
\]

Given a set $A$, we denote by $\chi_{A}$ the characteristic function
of $A$.
\end{notation}

\subsection{Non-compensation arguments\label{subsec:Non-compensation-arguments}}

In this section we prove a result (Proposition \ref{prop: non compensation powers}
below) which is a crucial ingredient of the proof of the Stability
Theorems \ref{thm Stability of C^K} and \ref{thm: C_M stability},
and of the study the asymptotics of the functions of our classes.
The statement of Proposition \ref{prop: non compensation powers}
is stronger than the result that we actually need here, since it involves
both purely imaginary powers and purely imaginary exponentials. However,
the full generality of this result will be used in a forthcoming paper,
in which we will study the Fourier transforms of power-constructible
functions.

We first recall the definition of continuously uniformly distributed
modulo 1 functions, which is a key ingredient in Proposition \ref{prop: non compensation powers}
(for the properties and uses of this notion, see \cite{kuipers_niederreiter:uniform_distributions_sequences}).
In what follows, $\mathrm{vol}_{i}$ stands for the Lebesgue measure
in $\mathbb{R}^{i}$, $i\ge1$.
\begin{defn}
\label{def CUD map} Let $\{x\}:=x-\lfloor x\rfloor$ be the fractional
part of the real number $x$ and let $F=\left(f_{1},\ldots,f_{\ell}\right)\colon[0,+\infty)\to\mathbb{R}^{\ell}$
be any map. If $I_{1},\ldots,I_{\ell}\subseteq\mathbb{R}$ are bounded
intervals with nonempty interior, we denote by $I$ the box $\prod_{j=1}^{\ell}I_{j}$.
For $T\ge0$, let %
\[
W_{F,I,T}:=\left\{ t\in[0,T]:\{F\left(t\right)\}\in I\right\} ,
\]
where $\{F(t)\}$ denotes the tuple $\left(\{f_{1}(t)\},\ldots,\{f_{\ell}\left(t\right)\}\right)$.

The map $F$ is said to be \emph{continuously uniformly distributed
modulo $1$}, in short \textbf{\emph{ }}\emph{c.u.d. mod $1$}, if
for every box $I\subseteq[0,1)^{\ell}$, 
\[
\lim_{T\to+\infty}\frac{\mathrm{vol}_{1}\left(W_{F,I,T}\right)}{T}=\mathrm{vol}_{\ell}\left(I\right).
\]
\end{defn}
We use the c.u.d. mod 1 property in the proof of the following proposition.
There, we deal with a family of complex exponential functions having
as phases the functions in the family $(\sigma_{j}\log(y)+p_{j}(y))_{j\in\{1,\ldots,n\}}$.
It turns out that in general we cannot extract from this family a
c.u.d. mod 1 subfamily, since $\log$ is not a c.u.d. mod 1 function
(although the family $\sigma_{j}\log y+p_{j}(y)$ is, whenever $p_{j}$
is not constant). To overcome this technical difficulty, we compose
$\sigma_{j}\log y+p_{j}(y)$ with the change of variables $y=\mathrm{e}^{t}$,
after which we are able to extract a c.u.d. mod 1 subfamily from the
family of phases $(\sigma_{j}t+p_{j}(\mathrm{e}^{t}))_{j\in\{1,\ldots,n\}}$. 
\begin{prop}
\label{prop: non compensation powers}Let $r\geq-1,\ b\geq1,\ \nu\in\mathbb{N},\ n\in\mathbb{N}\setminus\left\{ 0\right\} ,\ c_{1},\ldots,c_{n}\in\mathbb{C}$,
$\sigma_{0},\ldots\sigma_{n}\in\mathbb{R}$ and $p_{1},\ldots,p_{n}\in\mathbb{R}[X]$
be such that $p_{j}(0)=0$ for $j=1,\ldots,n$. Suppose that $\sigma_{j}\log(y)+p_{j}(y)\not=\sigma_{k}\log(y)+p_{k}(y)$
for $j\not=k$, and let 
\[
f\left(y\right)=y^{r}\left(\log y\right)^{\nu}\sum_{j=1}^{n}c_{j}y^{\mathrm{i}\sigma_{j}}\mathrm{e}^{\mathrm{i}p_{j}(y)}.
\]
The following statements hold.
\begin{enumerate}
\item \label{item:cud1} If $f\in L^{1}\left(\left(b,+\infty\right)\right)$
then $c_{j}=0$ for all $j=1,\ldots,n$. 
\item \label{item:cud2} Let $E(y)=\sum_{j=1}^{n}c_{j}y^{\mathrm{i}\sigma_{j}}\mathrm{e}^{\mathrm{i}p_{j}(y)}$,
where for at least one $j\in\{1,\ldots,n\}$ we have $c_{j}\not=0$
and $\sigma_{j}\log(y)+p_{j}(y)\not=0$. There exist $\varepsilon>0$
and a sequence of real numbers $(y_{m})_{m\in\mathbb{N}}$ which tends
to $+\infty$, such that for all $m\ge0$, $\vert E(y_{m})\vert\ge\varepsilon$.
\item \label{item:cud3} There exist $\delta>0$ and two sequences of real
numbers $(y_{1,m})_{m\in\mathbb{N}}$, $(y_{2,m})_{m\in\mathbb{N}}$
which both tend to $+\infty$, such that for all $m\ge0$, $\vert E(y_{1,m})-E(y_{2,m})\vert\ge\delta$. 
\end{enumerate}
\end{prop}
\begin{proof}
We may assume that at least one of the functions $g_{j}\left(y\right)=\sigma_{j}\log(y)+p_{j}(y)$
is not constant. Indeed, since $p_{j}(0)=0$, if $g_{j}$ is constant
then it is zero, hence in this case $n=1$ and $f$ is not integrable
unless $c_{1}=0$. Therefore we may assume without loss of generality
that $g_{1}$ is not constant, and that $c_{1}\not=0$.

If $f\in L^{1}\left(\left(b,+\infty\right)\right)$, the following
integral is finite for all $x$ such that $e^{x}\ge b$: 
\[
I(x):=\int_{b}^{\mathrm{e}^{x}}\left|f(y)\right|\ dy.
\]
Performing the change of variables $t=\log(y)$ we obtain 
\[
I(x)={\displaystyle \int_{\log(b)}^{x}t^{\nu}\mathrm{e}^{t(r+1)}\left|\sum_{j=1}^{n}c_{j}\mathrm{e}^{\mathrm{i}\sigma_{j}t}\mathrm{e}^{\mathrm{i}p_{j}(\mathrm{e}^{t})}\right|\ \mathrm{d}t.}
\]
Set $\varphi(t)=\sum_{j=1}^{n}c_{j}\mathrm{e}^{\mathrm{i}\sigma_{j}t}\mathrm{e}^{\mathrm{i}p_{j}(\mathrm{e}^{t})}$
and $f_{j}(t)=\sigma_{j}t+p_{j}(\mathrm{e}^{t})$ for $j=1,\ldots,n$.
Assume that $f_{1},\ldots,f_{\ell}$, for $\ell\le n$, is a basis
over $\mathbb{Q}$ of the $\mathbb{Q}$-vector space generated by
$f_{1},\ldots,f_{n}$. We write 
\[
f_{k}=r_{k,1}f_{1}+\cdots+r_{k,\ell}f_{\ell},\text{ for }k=\ell+1,\ldots,n,
\]
we denote by $\rho_{j}$ the least common multiple of the denominators
of $r_{\ell+1,j},\ldots,r_{n,j}$, and we set $\tilde{f_{j}}:=f_{j}/2\pi\rho_{j}$,
for $j=1,\ldots,\ell$ (note that this family is still $\mathbb{Q}$-linearly
independent). Then 
\[
f_{k}=2\pi m_{k,1}\tilde{f_{1}}+\cdots+2\pi m_{k,\ell}\tilde{f_{\ell}},\text{ for }k=\ell+1,\ldots,n,
\]
for some $m_{k,1},\ldots,m_{k,\ell}\in\mathbb{Z}$, and 
\[
\varphi(t)=P(\mathrm{e}^{2\pi\mathrm{i}\tilde{f_{1}}(t)},\ldots,\mathrm{e}^{2\pi\mathrm{i}\tilde{f_{\ell}}(t)}),
\]
where $P\in\mathbb{C}\left[X_{1},\ldots,X_{\ell},X_{1}^{-1},\ldots,X_{\ell}^{-1}\right]$
is a Laurent polynomial.

Note that $P$ contains at least the monomial $c_{1}X_{1}$ (we can
always choose $f_{1}$ as an element of our basis, since $c_{1}\not=0$
and $g_{1}\not=0$). Moreover since by hypothesis $f_{j}(t)\not=f_{k}(t)$,
(as functions) for $j\not=k$, the monomials of $P$ cannot cancel
out. It follows that $P$ is not constant, and therefore the algebraic
set $V:=\{P=0\}$ does not contain the torus $\mathbb{T}:=(S^{1})^{\ell}$.
By continuity of $P$, we can find a real number $\varepsilon>0$
and intervals $A_{j}^{\varepsilon}\subset[0,1)$, $j=1,\ldots,\ell$,
such that $\vert\varphi(t)\vert\ge\varepsilon$ on the set 
\[
W_{\varepsilon}=\left\{ t\ge\log(b):\ \left\{ \tilde{f_{j}}(t)\right\} \in A_{j}^{\varepsilon},\ j=1\ldots,\ell\right\} .
\]

We claim that the map $F=\left(\tilde{f_{1}},\ldots,\tilde{f_{\ell}}\right)$
is c.u.d. mod 1 (which implies in particular that $W_{\varepsilon}$
is nonempty). For this, we use the criterion \cite[Theorem 9.9]{kuipers_niederreiter:uniform_distributions_sequences},
i.e. we show that for any $h\in\mathbb{Z}^{\ell}$ such that $h\not=0$,
\[
\lim_{T\to+\infty}\frac{1}{T}\int_{1}^{T}\mathrm{e}^{2\pi\mathrm{i}\langle h,F(t)\rangle}\ \mathrm{d}t=0.
\]
We prove in fact that there exists $T_{0}\geq1$ such that 
\[
J(T)=\int_{T_{0}}^{T}\mathrm{e}^{2\pi\mathrm{i}\langle h,F(t)\rangle}\ \mathrm{d}t
\]
is bounded from above. For $h\in\mathbb{Z}^{\ell}$ such that $h\not=0$,
we can write $\langle h,F(t)\rangle=\sigma t+p(\mathrm{e}^{t})$,
with $\sigma\in\mathbb{R}$ and $p\in X\mathbb{R}[X]$. Since the
components of $F$ are $\mathbb{Q}$-linearly independent, $\sigma t+p(\mathrm{e}^{t})$
is not identically zero (equivalently, not constant). We can assume
that $p\not=0$, since if not, then $J(T)$ is clearly bounded, and
we are done. Let us write 
\[
\rho(t)=\frac{\langle h,F(t)\rangle}{a_{d}}=\mathrm{e}^{dt}+\frac{a_{d-1}}{a_{d}}\mathrm{e}^{(d-1)t}+\ldots+\frac{\sigma}{a_{d}}t,
\]
for some $d\ge1$, $a_{i}\in\mathbb{R}$ and $a_{d}\in\mathbb{R}\setminus\{0\}$.
Fix $T_{0}$ sufficiently large so that $t\mapsto\rho$ and $t\mapsto\rho'(t)$
are strictly increasing (to $+\infty$) on $[T_{0},+\infty)$, and
perform the change of variables $u=\rho(t)$ in $J(T)$ to obtain
\[
J(T)=\int_{T_{0}}^{T}\mathrm{e}^{2\pi\mathrm{i}a_{d}\rho(t)}\ \mathrm{d}t=\int_{\rho(T_{0})}^{\rho(T)}\dfrac{\mathrm{e}^{2\pi\mathrm{i}a_{d}u}}{\rho'(\rho^{-1}(u))}\ \mathrm{d}u.
\]
By the second mean value theorem for integrals applied to the real
part of $J(T)$, we have 
\[
\Re(J(T))=\dfrac{1}{{\rho'(T_{0})}}\int_{\rho(T_{0})}^{\tau}\cos(2\pi a_{d}u)\ \mathrm{d}u,
\]
for some $\tau\in(\rho(T_{0}),\rho(T)]$. This shows that the real
part of $J(T)$ is bounded from above, and so is the imaginary part
of $J(T)$ by the same computation.

Therefore $F$ is c.u.d. mod 1 and hence, by definition, the set $W_{\varepsilon}$
has infinite measure. Since 
\[
I(x)\ge\varepsilon\int_{[\log(b),x]\cap W_{\varepsilon}}t^{\nu}\mathrm{e}^{(r+1)t}\ \mathrm{d}t
\]
and $\nu\geq0$, $r\geq-1$, this implies that 
\[
\int_{b}^{+\infty}f\left(y\right)\text{d}y=\lim_{x\longrightarrow+\infty}I(x)=+\infty,
\]
and proves \eqref{item:cud1}.

To prove \eqref{item:cud2} and \eqref{item:cud3} we may still assume
that $c_{1}\not=0$ and $g_{1}\not=0$, by our hypothesis on $E$.
In this situation, since we have shown that $W_{\varepsilon}$ has
infinite measure, one can find a sequence $(t_{m})_{m\in\mathbb{N}}$
which tends to $+\infty$, such that for all $m\ge0$, $t_{m}\in W_{\varepsilon}$,
and therefore $\vert\varphi(t_{m})\vert\ge\varepsilon$. We set, for
all $m\in\mathbb{N}$, $y_{m}=\text{e}^{t_{m}}$, and we obtain $y_{m}\to+\infty$
and $E(y_{m})=\varphi(t_{m})$, which proves \eqref{item:cud2}.

We proceed in the same way to prove \eqref{item:cud3}. Since $P$
is not constant on $\mathbb{T}$, by continuity of $P$ one can find
$\delta>0$ and intervals $A_{j}^{\delta},B_{j}^{\delta}\subset[0,1)$,
$j=1,\ldots,\ell$, such that $\vert\varphi(t)-\varphi(t')\vert\ge\delta$
for any $t,t'$ such that $t\in A^{\delta}:=\{u\in\mathbb{R},\{\tilde{f}_{j}(u)\}\in A_{j}^{\delta},j=1,\ldots,\ell\}$
and $t'\in B^{\delta}:=\{u\in\mathbb{R},\{\tilde{f}_{j}(u)\}\in B_{j}^{\delta},j=1,\ldots,\ell\}$.
But since $F$ is c.u.d. mod 1, one can find two sequences $(t_{1,m})_{m\in\mathbb{N}}$
and $(t_{2,m})_{m\in\mathbb{N}}$ tending to $+\infty$, such that
for all $m\in\mathbb{N}$, $t_{1,m}\in A^{\delta}$ and $t_{2,m}\in B^{\delta}$.
Finally, we set $y_{1,m}=\text{e}^{t_{1,m}}$ and $y_{2,m}=\text{e}^{t_{2,m}}$
to obtain \eqref{item:cud3}. 
\end{proof}

\subsection{Properties of parametric strong functions\label{subsec:Results-about-parametric}}

In this section we give some examples of parametric strong functions
and list their properties. The results in this section are stated
for $\mathcal{E}$ and $\mathcal{A}$ for simplicity, but they also
hold for $\mathcal{E}_{\mathbb{K}}$ and $\mathcal{A}_{\mathbb{K}}$. 
\begin{examples}
\label{ex: unit^s}All (finite sum of finite products) of the following
functions are parametric strong functions (i.e. they belong to $\mathcal{A}$).
\begin{itemize}
\item Any subanalytic strong function (as in Definition \ref{def: classical strong}),
clearly.
\item $\left(U\left(x\right)\right)^{s}$, where $U\in\mathcal{S}\left(X\right)$
is a subanalytic strong unit of the form $U\left(x\right)=1+F\circ\psi\left(x\right)$,
with $\psi\in\mathcal{S}_{c}\left(X\right)$, $F\in\mathbb{R}\left\llbracket Z\right\rrbracket $
and $\sup_{z\in D^{N}}\left|F\left(z\right)\right|<1$ (where $D^{N}$
is the closed polydisk in $\mathbb{R}^{N}$ of radius $\frac{3}{2}$).
To see this, notice that the series $\tilde{F}=\left(1+F\left(Z\right)\right)^{s}=\sum{s \choose i}\left(F\left(Z\right)\right)^{i}\in\mathcal{E}\left\llbracket Z\right\rrbracket $
is strongly convergent (without poles) and $\left(U\left(x\right)\right)^{s}=\tilde{F}\circ\left(s,\psi\left(x\right)\right)$.
\item Let $B=\left\{ \left(x,y\right)\in\left(2,+\infty\right)\times\mathbb{R}:\ y>x\right\} $
and $\Phi\left(s,x,y\right)={\displaystyle \sum_{i\geq2}}\xi_{i}\left(s\right)\left(\frac{x}{y}\right)^{i}\in\mathcal{A}\left(B\right)$.
Then $\varphi\left(s,x\right):={\displaystyle \int_{x^{2}}^{+\infty}}\Phi\left(s,x,y\right)\text{d}y\in\mathcal{A}\left(\left(2,+\infty\right)\right)$.
To see this, integrate term by term (which is possible, since the
series $\Sigma_{i}\xi_{i}\left(s\right)Z^{i}$ is strongly convergent)
and find that $\varphi\left(s,x\right)=\sum_{i\geq0}\frac{\xi_{i+2}\left(s\right)}{i+1}x^{-i}$,
which is again a strongly convergent series with coefficients in $\mathcal{E}$,
composed with the $1$-bounded subanalytic function $x^{-1}\in\mathcal{S}\left(\left(2,+\infty\right)\right)$.
\end{itemize}
\end{examples}
\begin{rem}
\label{rem: analyticity in x and summability}$\ $
\begin{itemize}
\item If $\Phi\in\mathcal{A}\left(X\right)$ has no poles outside $P$,
then clearly for every fixed $s\in\left(\mathbb{C}\setminus P\right)$,
$x\longmapsto\Phi\left(s,x\right)$ is a complex-valued subanalytic
strong function (in the sense of Definition \eqref{def: classical strong}).
In particular, up to decomposing $X$ into subanalytic cells, we may
suppose that $\Phi$ depends analytically on $x$.
\item If $\Phi\left(s,x\right)=F\circ\left(s,\psi\left(x\right)\right)\in\mathcal{A}\left(X\right)$
is a parametric strong function then 
\[
\left\{ \xi_{I}\left(s\right)\psi\left(x\right)^{I}:\ I\in\mathbb{N}^{N}\right\} 
\]
 is a\emph{ normally summable} family of functions: the family $\left\{ \sup_{x\in X}\left|\xi_{I}\left(s\right)\psi\left(x\right)^{I}\right|\right\} \subseteq\left[0,1\right]$
is summable. In particular, if $\tilde{F}$ is obtained from $F$
by taking the sum only over some subset of the support of $F$ and
rearranging the terms, then $\tilde{F}\circ\left(s,\psi\left(x\right)\right)$
is a parametric strong function (without poles outside the set $P\left(F\right)$). 
\end{itemize}
\end{rem}
\begin{rem}
\label{rem: nested present}Let $\left(Z,Y\right)$ be an $\left(N+M\right)$-tuple
of variables and $F\left(s;Z,Y\right)=\sum_{I,J}\xi_{I,J}\left(s\right)Z^{I}Y^{J}\in\mathcal{E}\left\llbracket Z,Y\right\rrbracket $
be a strongly convergent series. Then, for all $J\in\mathbb{N}^{M}$,
the series $F_{J}:=\sum_{I}\xi_{I,J}\left(s\right)Z^{I}\in\mathcal{E}\left\llbracket Z\right\rrbracket $
is strongly convergent. Moreover, for every $1$-bounded subanalytic
map $c:X\longrightarrow\mathbb{R}^{N}$, we have $\xi_{J}^{c}\left(s,x\right):=F_{J}\circ\left(s,c\left(x\right)\right)\in\mathcal{A}\left(X\right)$.
Furthermore the series $F_{c}:=\sum_{J}\xi_{J}^{c}\left(s,x\right)Y^{J}\in\mathcal{A}\left(X\right)\left\llbracket Y\right\rrbracket $
is \emph{strongly convergent}, in the sense that the family $\left\{ \xi_{J,x}\left(s\right):=\xi_{J}^{c}\left(s,x\right)\right\} _{J,x}\subseteq\mathcal{E}$
has a non-accumulating set of common poles with bounded order and
the series $\sum_{J}\xi_{J}^{c}\left(s,x\right)Y^{J}$ defines a function
on $\left(\mathbb{C}\setminus P\right)\times X\times D^{M}$ which
is meromorphic in $s$ and analytic in $\left(x,Y\right)$.

It follows that, for every $1$-bounded map $\gamma:X\longrightarrow\mathbb{R}^{M}$,
the parametric strong function $\Phi\left(s,x\right):=F\circ\left(s,c\left(x\right),\gamma\left(x\right)\right)$
can also be written as the strongly convergent power series $F_{c}$
(with suitable parametric strong functions as coefficients), evaluated
at $Y=\gamma\left(x\right)$. We call $F_{c}\circ\left(s,\gamma\left(x\right)\right)=\sum_{J}\xi_{J}^{c}\left(s,x\right)\left(\gamma\left(x\right)\right)^{J}$
a \emph{nested presentation} of $\Phi$.
\end{rem}
We will often apply the above to the following situation: let $B\subseteq\mathbb{R}^{m+1}$
a subanalytic set such that the projection onto the first $m$ coordinates
of $B$ is $X$. Fix coordinates $\left(x,y\right)$, where $x$ is
an $m$-tuple and $y$ is a single variable. Suppose that $\left(c\left(x\right),\gamma\left(x,y\right)\right)$
is a $1$-bounded subanalytic map on $B$, where the first component
only depends on the variables $x$. Then the nested presentation of
$F\circ\left(s,c\left(x\right),\gamma\left(x,y\right)\right)\in\mathcal{A}\left(B\right)$
is of the form 
\begin{equation}
F_{c}\circ\left(s,\gamma\left(x,y\right)\right)=\sum_{J}\xi_{J}^{c}\left(s,x\right)\left(\gamma\left(x,y\right)\right)^{J},\label{eq:nested form}
\end{equation}
where the coefficients $\xi_{J}^{c}$ now belong to $\mathcal{A}\left(X\right)$.
\begin{rem}
\label{rem: rem on suban strong}Let $s$ be a fixed real or complex
number. Then Examples \ref{ex: unit^s}, Remarks \ref{rem: analyticity in x and summability}
and \ref{rem: nested present} also apply to real- or complex-valued
subanalytic strong functions.
\end{rem}

\subsection{Subanalytic preparation\label{subsec:Subanalytic-preparation}}

Let $\mathbb{K}\subseteq\mathbb{C}$ be a subfield and recall that
$\mathbb{F}_{\mathbb{K}}$ is $\mathbb{R}$ if $\mathbb{K}\subseteq\mathbb{R}$
and $\mathbb{C}$ otherwise.
\begin{defn}
\label{def: y-prep psi}Let $X\subseteq\mathbb{R}^{m}$ be a subanalytic
cell and
\begin{equation}
B=\left\{ \left(x,y\right):\ x\in X,\ a\left(x\right)<y<b\left(x\right)\right\} ,\label{eq:A_theta=00003DB}
\end{equation}
where $a,b:X\longrightarrow\mathbb{R}$ are analytic subanalytic functions
with $1\leq a\left(x\right)<b\left(x\right)\ \text{for all }x\in X$,
and $b$ is allowed to be $\equiv+\infty$. We say that $B$ has\emph{
bounded $y$-fibres} if $b<+\infty$ and \emph{unbounded $y$-fibres}
if $b\equiv+\infty$.
\begin{itemize}
\item A $1$-bounded subanalytic map $\psi:B\longrightarrow\mathbb{R}^{M+2}\in\mathcal{S}_{c}^{M+2}\left(B\right)$
is \emph{$y$-prepared} if it has the form
\begin{equation}
\psi\left(x,y\right)=\left(c\left(x\right),\left(\frac{a\left(x\right)}{y}\right)^{\frac{1}{d}},\left(\frac{y}{b\left(x\right)}\right)^{\frac{1}{d}}\right),\label{eq: psi}
\end{equation}
where $d\in\mathbb{N}$. \\
If $b\equiv+\infty$, then we will implicitly assume that the last
component is missing and hence $\psi:B\longrightarrow\mathbb{R}^{M+1}$.
\item An $\mathbb{F}_{\mathbb{K}}$-valued subanalytic strong function $W:B\longrightarrow\mathbb{F}_{\mathbb{K}}$
is \emph{$\psi$-prepared} if $\psi$ is a $y$-prepared $1$-bounded
subanalytic map as in \eqref{eq: psi} and
\[
W\left(x,y\right)=F\circ\psi\left(x,y\right),
\]
for some power series $F\in\mathbb{F}_{\mathbb{K}}\left\llbracket Z\right\rrbracket $
which converges in a neighbourhood of the ball of radius $\frac{3}{2}$.
Notice that $W$ has also a nested presentation (see \ref{rem: nested present})
as a strongly convergent power series with coefficients $\mathbb{F}_{\mathbb{K}}$-valued
subanalytic strong functions on $X$, evaluated at $\gamma\left(x,y\right)=\left(\left(\frac{a\left(x\right)}{y}\right)^{\frac{1}{d}},\left(\frac{y}{b\left(x\right)}\right)^{\frac{1}{d}}\right)$
\item A subanalytic function $f\in\mathcal{S}\left(B\right)$ is \emph{prepared}
if there are $\nu\in\mathbb{Z}$, an analytic function $f_{0}\in\mathcal{S}\left(X\right)$
and a $\psi$-prepared real-valued subanalytic strong unit $U$ (for
some $\psi$ as in \eqref{eq: psi}) such that
\[
f\left(x,y\right)=f_{0}\left(x\right)y^{\frac{\nu}{d}}U\left(x,y\right)
\]
\end{itemize}
\end{defn}
Let us recall some notation from \cite[Definitions 3.2, 3.3, 3.4 and 3.8]{ccmrs:integration-oscillatory}.
In particular, $A\subseteq\mathbb{R}^{m+1}$ will be a cell open over
$\mathbb{R}^{m}$ (it will always be possible to suppose that the
base of $A$ is $X\subseteq\mathbb{R}^{m}$) with analytic subanalytic
centre $\theta_{A}$ and such that the set $I_{A}:=\left\{ y-\theta_{A}\left(x\right):\ \left(x,y\right)\in A\right\} $
is contained in one of the sets $\left(-\infty,-1\right),\left(-1,0\right),\left(0,1\right),\left(1,+\infty\right)$,
as in \cite[Definition 3.4]{ccmrs:integration-oscillatory}. We now
perform a change of coordinates with the aim of mapping the set $I_{A}$
to the interval $\left(1,+\infty\right)$: there are unique sign conditions
$\sigma_{A},\tau_{A}\in\left\{ -1,1\right\} $ such that
\begin{equation}
A=\left\{ \left(x,y\right):\ x\in X,\ a_{A}\left(x\right)<\sigma_{A}\left(y-\theta_{A}\left(x\right)\right)^{\tau_{A}}<b_{A}\left(x\right)\right\} \label{eq:A}
\end{equation}
for some analytic subanalytic functions $a_{A},b_{A}$ such that $1\leq a_{A}\left(x\right)<b_{A}\left(x\right)\leq+\infty$.
Let
\begin{equation}
B_{A}=\left\{ \left(x,y\right):\ x\in X,\ a_{A}\left(x\right)<y<b_{A}\left(x\right)\right\} \label{eq: Atheta}
\end{equation}
and $\Pi_{A}:B_{A}\longrightarrow A$ be the bijection
\begin{equation}
\Pi_{A}\left(x,y\right)=\left(x,\sigma_{A}y^{\tau_{A}}+\theta_{A}\left(x\right)\right),\ \Pi_{A}^{-1}\left(x,y\right)=\left(x,\sigma_{A}\left(y-\theta_{A}\left(x\right)\right)^{\tau_{A}}\right).\label{eq:Ptheta}
\end{equation}
We will still denote by $\Pi_{A}$ the map $\mathbb{C}\times B_{A}\ni\left(s,x,y\right)\longmapsto\left(s,\Pi_{A}\left(x,y\right)\right)\in\mathbb{C}\times A$.
\begin{rem}
\label{rem: cell at infty}By \cite[Definition 3.4(3)]{ccmrs:integration-oscillatory},
if $A$ is a cell of the form $A=\{\left(x,y\right):\ x\in X,\ y>f\left(x\right)\}$,
then $\sigma_{A}=\tau_{A}=1$ and $\theta_{A}=0$. Hence in this case
$a_{A}=f,\ b_{A}=+\infty$ and $B_{A}=A$.
\end{rem}
\begin{prop}
\label{prop: suban prep}\cite{lr:prep}\cite[Remark 3.12]{ccmrs:integration-oscillatory}Let
$\mathcal{F}\subseteq\mathcal{S}\left(X\times\mathbb{R}\right)$ be
a finite collection of subanalytic functions. There is a cell decomposition
of $\mathbb{R}^{m+1}$ compatible with $X$ such that for each cell
$A$ that is open over $\mathbb{R}^{m}$ (which we may suppose to
be of the form \eqref{eq:A}) and for every $h\in\mathcal{F}$, $h\circ\Pi_{A}$
is prepared on $B_{A}$.
\end{prop}

\section{Preparation of (parametric) power-constructible functions\label{sec:Preparation of powers}}

Let $\mathbb{K}\subseteq\mathbb{C}$ be a subfield and recall that
$\mathbb{F}_{\mathbb{K}}$ is $\mathbb{R}$ if $\mathbb{K}\subseteq\mathbb{R}$
and $\mathbb{C}$ otherwise. In this section, $y$ will be a single
variable. For each of the classes introduced in Sections \ref{subsec:power-constructible}
and \ref{subsec:Parametric powers}, we will give a \emph{prepared
presentation} of its elements, with respect to the last subanalytic
variable (denoted by $y$).

\subsection{\label{subsec:prep of power-constr}Preparation of power-constructible
functions}
\begin{defn}
\label{def: prepared generator of Cpow}Let $B$ be as in \eqref{eq:A_theta=00003DB}.
A generator $T$ of the $\mathbb{F}_{\mathbb{K}}$-algebra $\mathcal{C}^{\mathbb{K}}\left(B\right)$
is called \emph{prepared }if 
\begin{equation}
T\left(x,y\right)=G_{0}\left(x\right)y^{\frac{\eta}{d}}\left(\log y\right)^{\mu}W\left(x,y\right),\label{eq: prep gen of Cpow}
\end{equation}
where $G_{0}\in\mathcal{C}^{\mathbb{K}}\left(X\right),\ \eta\in\mathbb{K},\ \mu\in\mathbb{N}$
and $W$ is a $\psi$-prepared $\mathbb{F}_{\mathbb{K}}$-valued subanalytic
strong function, for some $1$-bounded $\psi$ as in \eqref{eq: psi}.
\end{defn}
It follows from Remark \ref{rem: rem on suban strong} that, in the
notation of \eqref{eq:nested form}, if $B$ has bounded $y$-fibres
(i.e. $b<+\infty$), then $W$ can be written as
\begin{equation}
\sum_{m,n}\xi_{m,n}^{c}\left(x\right)\left(\frac{a\left(x\right)}{y}\right)^{\frac{m}{d}}\left(\frac{y}{b\left(x\right)}\right)^{\frac{n}{d}},\label{eq: bounded Cpow W}
\end{equation}
and if $B$ has unbounded $y$-fibres (i.e. $b\equiv+\infty$), then
$W$ can be written as
\begin{equation}
\sum_{k}\xi_{k}^{c}\left(x\right)\left(\frac{a\left(x\right)}{y}\right)^{\frac{k}{d}}.\label{eq:unbounded Cpow W}
\end{equation}

\begin{prop}
\label{prop: prep of power-constr}Let $\mathcal{F}\subseteq\mathcal{C}^{\mathbb{K}}\left(X\times\mathbb{R}\right)$
be a finite collection of $\mathbb{K}$-power-constructible functions.
Then there is a cell decomposition of $\mathbb{R}^{m+1}$ compatible
with $X$ such that for each cell $A$ that is open over $\mathbb{R}^{m}$
(which we may suppose to be of the form \eqref{eq:A}) and each $h\in\mathcal{F}$,
$h\circ\Pi_{A}$ is a finite sum of prepared generators of the form
\eqref{def: prepared generator of Cpow}.
\end{prop}
\begin{proof}
The proof is a straightforward refinement of the proofs of \cite[Corollary 3.5]{cluckers_miller:loci_integrability}
and \cite[Proposition 3.10]{ccmrs:integration-oscillatory}: one prepares
first all the subanalytic data appearing in $h$, by Proposition \ref{prop: suban prep},
and then observes the effect of applying $\log$ or a power $\eta\in\mathbb{K}$
to a subanalytic prepared function. In particular, notice that if
$U\left(x,y\right)$ is a $\psi$-prepared subanalytic strong unit,
then $U^{\eta}$ is again a $\psi$-prepared $\mathbb{F}_{\mathbb{K}}$-valued
subanalytic strong unit.
\end{proof}

\subsection{\label{subsec:prep of strong}Preparation of parametric strong functions}

Let $\mathbb{K}\subseteq\mathbb{C}$ be a subfield and refer to the
definitions of $\mathcal{E}_{\mathbb{K}},\mathcal{A}_{\mathbb{K}}$
in Section \ref{subsubsec:Variants}.
\begin{defn}
\label{def: prep param strong}Let $B$ be as in \eqref{eq:A_theta=00003DB}.
A parametric strong function $\Phi\in\mathcal{A}_{\mathbb{K}}\left(B\right)$
is called \emph{$\psi$-prepared} (where $\psi$ is as in \eqref{eq: psi})
if there exists a strongly convergent series $F=\sum\xi_{I}\left(s\right)Z^{I}\in\mathcal{E}_{\mathbb{K}}\left\llbracket Z\right\rrbracket $
such that 
\begin{equation}
\forall\left(s,x,y\right)\in\left(\mathbb{C}\setminus P\left(F\right)\right)\times B,\ \Phi\left(s,x,y\right)=F\circ\left(s,\psi\left(x,y\right)\right).\label{eq: prep param strong}
\end{equation}
\end{defn}
Notice that if $\Phi$ is $\psi$-prepared, then $\Phi$ has a nested
presentation (see Remark \ref{rem: nested present}) as a power series
with coefficients in $\mathcal{A}_{\mathbb{K}}\left(X\right)$, evaluated
at $\gamma\left(x,y\right)=\left(\left(\frac{a\left(x\right)}{y}\right)^{\frac{1}{d}},\left(\frac{y}{b\left(x\right)}\right)^{\frac{1}{d}}\right)$:
\begin{equation}
\forall\left(s,x,y\right)\in\mathbb{C}\setminus P\left(F\right)\times B,\ \Phi\left(s,x,y\right)=\sum_{m,n}\xi_{m,n}^{c}\left(s,x\right)\left(\frac{a\left(x\right)}{y}\right)^{\frac{m}{d}}\left(\frac{y}{b\left(x\right)}\right)^{\frac{n}{d}},\label{eq:nested prep param strong}
\end{equation}
where $\xi_{m,n}^{c}\left(s,x\right)=\sum_{J}\xi_{J,m,n}\left(s\right)\left(c\left(x\right)\right)^{J}\in\mathcal{A}_{\mathbb{K}}\left(X\right)$.
\begin{rem}
\label{rem: splitting of param}Let $\Phi\in\mathcal{A}_{\mathbb{K}}\left(B\right)$
be $\psi$-prepared, as above. If $B$ has unbounded $y$-fibres (i.e.
$b\equiv+\infty$ in \eqref{eq:A_theta=00003DB}), recall that 
\begin{equation}
\psi\left(x,y\right)=\left(c\left(x\right),\left(\frac{a\left(x\right)}{y}\right)^{\frac{1}{d}}\right),\label{eq: psi unbounded}
\end{equation}
hence the nested $\psi$-prepared form of $\Phi$ is
\begin{equation}
\forall\left(s,x,y\right)\in\left(\mathbb{C}\setminus P\left(F\right)\right)\times B,\ \Phi\left(s,x,y\right)=\sum_{k}\xi_{k}^{c}\left(s,x\right)\left(\frac{a\left(x\right)}{y}\right)^{\frac{k}{d}},\label{eq: nested unbounded}
\end{equation}
where $\xi_{k}^{c}\left(s,x\right)=\sum_{J}\xi_{J,k}\left(s\right)\left(c\left(x\right)\right)^{J}\in\mathcal{A}_{\mathbb{K}}\left(X\right)$.
\end{rem}
\begin{lem}
\label{lem: prep of strong}Let $\mathcal{F}\subseteq\mathcal{A}_{\mathbb{K}}\left(X\times\mathbb{R}\right)$
be a finite set of functions $\Phi$ which have no poles outside some
closed discrete set $P\subseteq\mathbb{K}$. Then there is a cell
decomposition of $\mathbb{R}^{m+1}$ compatible with $X$ such that
for each cell $A$ that is open over $\mathbb{R}^{m}$ (which we may
suppose to be of the form \eqref{eq:A}), each $\Phi\circ\Pi_{A}$
is $\psi$-prepared on $\left(\mathbb{C}\setminus P\right)\times B_{A}$
(for some $y$-prepared $1$-bounded subanalytic map $\psi$ as in
\eqref{eq: psi}).
\end{lem}
\begin{proof}
We will consider the case of a single function $\Phi$ for simplicity
of notation (the general case is obtained by taking as $\Phi$ the
product of the functions in $\mathcal{F}$). Write $\Phi=G\circ\left(s,\eta\right)$,
where $G=\sum_{I}\varphi_{I}\left(s\right)T^{I}\in\mathcal{E}\left\llbracket T\right\rrbracket $
is a strongly convergent series in $N$ variables $T$ and $\eta=\left(\eta_{1},\ldots,\eta_{N}\right):X\times\mathbb{R}\longrightarrow\mathbb{R}^{N}$
is a $1$-bounded subanalytic map.

Apply subanalytic preparation (Proposition \eqref{prop: suban prep})
to the components of $\eta$. This yields a cell decomposition of
$X\times\mathbb{R}$ such that, if $A$ is a cell of the form \eqref{eq:A},
then the components of $\eta\circ\Pi_{A}$ are $\hat{\psi}$-prepared
$B_{A}$, where $\hat{\psi}\left(x,y\right)=\left(\hat{c}\left(x\right),\left(\frac{a_{A}\left(x\right)}{y}\right)^{\frac{1}{d}},\left(\frac{y}{b_{A}\left(x\right)}\right)^{\frac{1}{d}}\right)$
is a $y$-prepared strongly subanalytic map:
\[
\eta_{j}\circ\Pi_{A}\left(x,y\right)=c_{j}\left(x\right)y^{\frac{\ell_{j}}{d}}v_{j}\left(x,y\right)\ \ \ \left(1\leq j\leq N\right),
\]
where $c_{j}\in\mathcal{S}\left(X\right)$ is analytic, $\ell_{j}$
is an integer and $v_{j}$ is a $\hat{\psi}$-prepared strong unit.
By rescaling the unit, we may furthermore assume that $\left|c_{j}\left(x\right)y^{\frac{\ell_{j}}{d}}\right|\leq1$
on the closure of $B_{A}$. Partition
\begin{align*}
\left\{ 1,\ldots,N\right\}  & =\bigcup_{*\in\left\{ <,=,>\right\} }J_{*}\\
 & =\bigcup_{*\in\left\{ <,=,>\right\} }\left\{ j:\ \frac{\ell_{j}}{d}*0\right\} 
\end{align*}
and notice that the subanalytic map $\tilde{c}:=\left(\tilde{c}_{1},\ldots,\tilde{c}_{N}\right)$
given by
\[
\tilde{c_{j}}\left(x\right):=\begin{cases}
c_{j}\left(x\right)\cdot\left(a_{A}\left(x\right)\right)^{\frac{\ell_{j}}{d}} & \left(j\in J_{<}\right)\\
c_{j}\left(x\right) & \left(j\in J_{=}\right)\\
c_{j}\left(x\right)\cdot\left(b_{A}\left(x\right)\right)^{\frac{\ell_{j}}{d}} & \left(j\in J_{>}\right)
\end{cases}
\]
is $1$-bounded. Hence,
\[
\eta_{j}\circ\Pi_{A}\left(x,y\right)=\begin{cases}
\tilde{c_{j}}\left(x\right)\left(\frac{a_{A}\left(x\right)}{y}\right)^{-\frac{\ell_{j}}{d}}v_{j}\left(x,y\right) & \left(j\in J_{<}\right)\\
\tilde{c_{j}}\left(x\right)v_{j}\left(x,y\right) & \left(j\in J_{=}\right)\\
\tilde{c_{j}}\left(x\right)\left(\frac{y}{b_{A}\left(x\right)}\right)^{\frac{\ell_{j}}{d}}v_{j}\left(x,y\right) & \left(j\in J_{>}\right)
\end{cases}
\]
and, for $I=\left(i_{1},\ldots,i_{N}\right)\in\mathbb{N}^{N}$,
\[
\left(\eta_{j}\circ\Pi_{A}\left(x,y\right)\right)^{i_{j}}=\tilde{c_{j}}\left(x\right)^{i_{j}}f_{I,j}\left(x,y\right),
\]
where
\[
f_{I,j}\left(x,y\right)=\begin{cases}
\left(\frac{a_{A}\left(x\right)}{y}\right)^{-\frac{\ell_{j}}{d}i_{j}}\left(v_{j}\left(x,y\right)\right)^{i_{j}} & \left(j\in J_{<}\right)\\
\left(v_{j}\left(x,y\right)\right)^{i_{j}} & \left(j\in J_{=}\right)\\
\left(\frac{y}{b_{A}\left(x\right)}\right)^{\frac{\ell_{j}}{d}i_{j}}\left(v_{j}\left(x,y\right)\right)^{i_{j}} & \left(j\in J_{>}\right)
\end{cases}.
\]

Notice that the $f_{I,j}$ are $\hat{\psi}$-prepared subanalytic
strong functions, hence so is their product $f_{I}\left(x,y\right):=\prod_{j\leq N}f_{I,j}\left(x,y\right)$.
Therefore, there is a strongly convergent power series with coefficients
in $\mathbb{F}_{\mathbb{K}}$ 
\[
F_{I}=\sum_{K,m,n}d_{K,m,n}^{I}\tilde{Z}^{K}Y_{1}^{m}Y_{2}^{n}\in\mathbb{F}_{\mathbb{K}}\left\llbracket \tilde{Z},Y_{1},Y_{2}\right\rrbracket 
\]
 such that $f_{I}\left(x,y\right)=F_{I}\circ\hat{\psi}\left(x,y\right)$.

Therefore, on $B_{A}$ we can write
\begin{align*}
 & \Phi\circ\Pi_{A}\left(s,x,y\right)=\\
 & =\sum_{I=\left(i_{1},\ldots,i_{N}\right)}\varphi_{I}\left(s\right)\left(\eta\circ\Pi_{A}\left(x,y\right)\right)^{I}\\
 & =\sum_{I=\left(i_{1},\ldots,i_{N}\right)}\varphi_{I}\left(s\right)\left(\tilde{c}\left(x\right)\right)^{I}f_{I}\left(x,y\right)\\
 & =\sum_{I=\left(i_{1},\ldots,i_{N}\right)}\varphi_{I}\left(s\right)\left(\tilde{c}\left(x\right)\right)^{I}\sum_{K,m,n}d_{K,m,n}^{I}\left(\hat{c}\left(x\right)\right)^{K}\left(\frac{a_{A}\left(x\right)}{y}\right)^{\frac{m}{d}}\left(\frac{y}{b_{A}\left(x\right)}\right)^{\frac{n}{d}}\\
 & =\sum_{I,K,m,n}d_{K,m,n}^{I}\varphi_{I}\left(s\right)\left(\tilde{c}\left(x\right)\right)^{I}\left(\hat{c}\left(x\right)\right)^{K}\left(\frac{a_{A}\left(x\right)}{y}\right)^{m}\left(\frac{y}{b_{A}\left(x\right)}\right)^{n}.
\end{align*}
Now, if we let $\tilde{I}=\left(I,K\right)$ and 
\[
\xi_{\tilde{I},m,n}\left(s\right)=d_{K,m,n}^{I}\varphi_{I}\left(s\right),
\]
then the family $\left\{ \xi_{\tilde{I},m,n}\right\} $ is strong
and the series
\[
F=\sum_{\tilde{I},m.n}\xi_{\tilde{I},m,n}\left(s\right)Z^{\tilde{I}}Y_{1}^{m}Y_{2}^{n}\in\mathcal{E}_{\mathbb{K}}\left\llbracket Z,Y_{1},Y_{2}\right\rrbracket 
\]
is strongly convergent (with $P\left(F\right)=P\left(G\right)$).
Let $c\left(x\right)=\left(\tilde{c}\left(x\right),\hat{c}\left(x\right)\right)$.
Then, in the notation of \eqref{eq: psi}, on $B_{A}$ we have 
\[
\Phi\circ\Pi_{A}\left(s,x,y\right)=F\circ\left(s,\psi\left(x,y\right)\right),
\]
so $\Phi\circ\left(s,\Pi_{A}\left(x,y\right)\right)$ is $\psi$-prepared
on $B_{A}$, as required.
\end{proof}

\subsection{\label{subsec:prep of mellin}Preparation of parametric power-constructible
functions}

In this section we let $\mathcal{D}$ be either $\mathcal{C}^{\mathbb{K},\mathcal{M}}$
or $\mathcal{C}^{\mathcal{P}\left(\mathbb{K}\right),\mathcal{M}}$
(see Section \ref{subsubsec:Variants}).
\begin{defn}
\label{def: prep gen of C^M}Let $B$ be as in \eqref{eq:A_theta=00003DB}
and $P\subseteq\mathbb{K}$ be a closed discrete set. A generator
$T\in\mathcal{D}\left(B\right)$ with no poles outside $P$ is \emph{prepared}
if for all $\left(s,x,y\right)\in\left(\mathbb{C}\setminus P\right)\times B$,
\begin{equation}
T\left(s,x,y\right)=G_{0}\left(s,x\right)y^{\frac{\ell s+\eta}{d}}\left(\log y\right)^{\mu}\Phi\left(s,x,y\right),\label{eq:prep gen CpowM}
\end{equation}
where $G_{0}\in\mathcal{D}\left(X\right),\ \ell,\eta\in\mathbb{\mathbb{K}},\ \mu\in\mathbb{N}$
and $\Phi\in\mathcal{A}_{\mathbb{K}}\left(B\right)$ is a $\psi$-prepared
parametric strong function (see Definition \ref{def: prep param strong}).
If $\mathcal{D}=\mathcal{C}^{\mathbb{K},\mathcal{M}}$, then we require
that $\ell\in\mathbb{Z}$.
\end{defn}
\begin{prop}
\label{prop: prep of C^M}Let $P\subseteq\mathbb{K}$ be a closed
discrete set and $h\in\mathcal{D}\left(X\times\mathbb{R}\right)$
have no poles outside $P$. Then there is a cell decomposition of
$\mathbb{R}^{m+1}$ compatible with $X$ such that for each cell $A$
that is open over $\mathbb{R}^{m}$ (which we may suppose to be of
the form \eqref{eq:A}), $h\circ\Pi_{A}$ is a finite sum of prepared
generators on $\left(\mathbb{C}\setminus P\right)\times B_{A}$.
\end{prop}
\begin{proof}
Suppose first that $\mathcal{D}=\mathcal{C}^{\mathbb{\mathcal{P}\left(K\right)},\mathcal{M}}$.
Write $h$ as a finite sum of generators of the form 
\[
T\left(s,x,y\right)=\Phi\left(s,x,y\right)\cdot g\left(x,y\right)\cdot f_{1}\left(x,y\right)^{\alpha_{1}s}\cdot\ldots\cdot f_{n}\left(x,y\right)^{\alpha_{n}s},
\]
with $\Phi\in\mathcal{A}_{\mathbb{K}}\left(X\times\mathbb{R}\right),\ g\in\mathcal{C}^{\mathbb{K}}\left(X\times\mathbb{R}\right),\ f_{i}\in\mathcal{S}_{+}\left(X\times\mathbb{R}\right),\ \alpha_{i}\in\mathbb{K}$.
Apply Proposition \ref{prop: suban prep} simultaneously to all the
$f_{i}$ and to all the subanalytic data in all the $\Phi$ and $g$
appearing in the generators. This yields a cell decomposition of $X\times\mathbb{R}$
such that on each cell $A$ with centre $\theta_{A}$, there is a
$y$-prepared subanalytic map $\psi$ as in \eqref{eq: psi} such
that, after composing with $\Pi_{A}$ all the subanalytic functions
considered above are prepared. In particular, each of the $f_{j}$
appearing in the parametric power, after composing with $\Pi_{A}$,
has the form
\[
\tilde{f}_{j}\left(x\right)y^{\frac{\ell_{j}}{d}}U_{j}\left(x,y\right),
\]
where $\tilde{f}_{j}\in\mathcal{S}_{+}\left(X\right),\ \ell_{j}\in\mathbb{Z}$
and $U_{j}\in\mathcal{S}\left(B_{A}\right)$ is a $\psi$-prepared
subanalytic strong unit. Hence, by the second of Examples \ref{ex: unit^s},
$\Xi_{j}\left(s,x,y\right):=\left|U_{j}\left(x,y\right)\right|^{s}\in\mathcal{A}_{\mathbb{K}}\left(B_{A}\right)$
and is $\psi$-prepared. 

Apply Proposition \ref{prop: prep of power-constr} to prepare each
$g\circ\Pi_{A}$, which can be hence written as a finite sum of terms
of the form
\[
g_{j}\left(x\right)y^{\frac{\eta_{j}}{d}}\left(\log y\right)^{\nu_{j}}W_{j}\left(x,y\right),
\]
where $\nu_{j}\in\mathbb{N},\ \eta_{j}\in\mathbb{K},\ g_{j}\in\mathcal{C}^{\mathbb{K}}\left(X\right)$
is analytic and $W_{j}$ is an $\mathbb{F}_{\mathbb{K}}$-valued $\psi$-prepared
subanalytic strong function on $B_{A}$.

Apply Lemma \ref{lem: prep of strong} to $\psi$-prepare each $\Phi\circ\Pi_{A}$
on $B_{A}$ as $F_{j}\circ\left(s,\psi\left(x,y\right)\right)$, where
$\psi$ has now some extra components depending only on the variables
$x$. Notice that this does not affect the preparation work already
done.

Finally, define $G_{j}\left(s,x\right)=\tilde{f}_{j}\left(x\right)^{\alpha_{j}s}g_{j}\left(x\right)$
and $\Phi_{j}\left(s,x,y\right)=F_{j}\circ\left(s,\psi\left(x,y\right)\right)\cdot W_{j}\left(x,y\right)\cdot\Xi_{j}\left(s,x,y\right)$.
Then clearly $G_{j}\in\mathcal{C}^{\mathbb{K},\mathcal{M}}\left(X\right)$
and $\Phi_{j}\in\mathcal{A}_{\mathbb{K}}\left(B_{A}\right)$ is $\psi$-prepared,
with no poles outside $P$, hence we have written $h\circ\Pi_{A}$
as a finite sum of terms of the form
\[
G_{j}\left(s,x\right)\cdot y^{\frac{\alpha_{j}\ell_{j}s+\eta_{j}}{d}}\cdot\left(\log y\right)^{\nu_{j}}\Phi_{j}\left(s,x,y\right)
\]
and we are done.

If $\mathcal{D}=\mathcal{C}^{\mathbb{K},\mathcal{M}}$, then repeat
the above proof with $n=\alpha_{1}=1$.
\end{proof}

\section{Integration of prepared (parametric) power-constructible generators\label{sec:Integration-of-gen of C^M}}

In this section we let $\mathcal{D}$ be either $\mathcal{C}^{\mathbb{K},\mathcal{M}}$
or $\mathcal{C}^{\mathcal{P}\left(\mathbb{K}\right),\mathcal{M}}$.

Given a cell $B\subseteq\mathbb{R}^{m+1}$, we study the integrability,
and compute the integral, of a prepared generator of $\mathcal{D}\left(B\right)$.

Let $B$ be as in \eqref{eq:A_theta=00003DB} and $T\in\mathcal{D}\left(B\right)$
be a prepared generator with no poles outside $P$ (for some discrete
closed set $P\subseteq\mathbb{K}$). We aim to study the nature of
the parametric integral
\begin{equation}
\int_{a\left(x\right)}^{b\left(x\right)}T\left(s,x,y\right)\text{d}y,\label{eq: integral of gen}
\end{equation}
for all $\left(s,x\right)\in\left(\mathbb{C}\setminus P\right)\times X$
such that $y\longmapsto T\left(s,x,y\right)\in L^{1}\left(B_{x}\right)$.

We prove that there exist a closed discrete set $P'\supseteq P$ and
a function $H\in\mathcal{D}\left(X\right)$ with no poles outside
$P'$ such that the above integral coincides with $H$.

We start by recalling the classical formula to compute the antiderivative
of any power-log monomial in $y$.
\begin{lem}
\label{lem: integral of a naive in y}Let $\ell,\gamma\in\mathbb{K},\ d,\mu\in\mathbb{N}$
with $\ell,d\not=0$. Let $s\in\mathbb{C}$ such that $\ell s+\gamma\not=-d$.
Then

\begin{equation}
\int y^{\frac{\ell s+\gamma}{d}}\left(\log y\right)^{\mu}\mathrm{d}y=\sum_{i=0}^{\mu}c_{\mu,i}\left(\log y\right)^{i}\frac{y^{\frac{\ell s+\gamma+d}{d}}}{\left(\ell s+\gamma+d\right)^{\mu+1-i}},\label{eq:antiderivative}
\end{equation}
where $c_{\mu,i}=\left(-1\right)^{\mu-i}\frac{\mu!}{i!}d^{\mu+1-i}$.
\end{lem}

\subsection{Cells with bounded $y$-fibres\label{subsec:Bounded-cells}}

Recall that $B$ is as in \eqref{eq:A_theta=00003DB} and suppose
that $b<+\infty$. Let $T\in\mathcal{D}\left(B\right)$ be a prepared
generator (as in \eqref{eq:prep gen CpowM}) without poles outside
some closed discrete set $P\subseteq\mathbb{K}$. We study the integrability
of $T$ on $B$: since $B$ has bounded $y$-fibres, the function
$y\longmapsto T\left(s,x,y\right)$ extends to a continuous function
on the boundary of $B_{x}$, hence the integral 
\[
\int_{a\left(x\right)}^{b\left(x\right)}T\left(s,x,y\right)\text{d}y
\]
 is finite. Let us compute it.

Let 
\[
P':=\begin{cases}
P\cup\left\{ s:\ \ell s+\eta\in\mathbb{Z}\right\}  & \text{if\ }\ell\not=0,\\
P & \text{if\ }\ell=0.
\end{cases}
\]
There are several cases to consider.

$\bullet$ If $\ell\not=0$, then, for $\left(s,x\right)\in\left(\mathbb{C}\setminus P'\right)\times X$,
we deduce from Lemma \ref{lem: integral of a naive in y} and normal
convergence that

\begin{align}
 & \int_{a\left(x\right)}^{b\left(x\right)}T\left(s,x,y\right)\text{d}y\label{eq:integral of bounded gen}\\
 & =\sum_{i=0}^{\mu}\sum_{m,n}\frac{c_{\mu,i}G_{0}\left(s,x\right)\xi_{m,n}\left(s,x\right)}{\left(\ell s+\eta+d-m+n\right)^{\mu+1-i}}\frac{\left(a\left(x\right)\right)^{\frac{m}{d}}}{\left(b\left(x\right)\right)^{\frac{n}{d}}}\cdot\left[y^{\frac{\ell s+\eta+d-m+n}{d}}\left(\log y\right)^{i}\right]_{a\left(x\right)}^{b\left(x\right)}\nonumber \\
 & =\sum_{i=0}^{\mu}c_{\mu,i}G_{0}\left(s,x\right)\left(b\left(x\right)\right)^{\frac{\ell s+\eta+d}{d}}\left(\log b\left(x\right)\right)^{i}\sum_{m,n}\frac{\xi_{m,n}\left(s,x\right)\left(\frac{a\left(x\right)}{b\left(x\right)}\right)^{\frac{m}{d}}}{\left(\ell s+\eta+d-m+n\right)^{\mu+1-i}}\nonumber \\
 & -\sum_{i=0}^{\mu}c_{\mu,i}G_{0}\left(s,x\right)\left(a\left(x\right)\right)^{\frac{\ell s+\eta+d}{d}}\left(\log a\left(x\right)\right)^{i}\sum_{m,n}\frac{\xi_{m,n}\left(s,x\right)\left(\frac{a\left(x\right)}{b\left(x\right)}\right)^{\frac{n}{d}}}{\left(\ell s+\eta+d-m+n\right)^{\mu+1-i}}\nonumber 
\end{align}

As a consequence of the Dominated Convergence Theorem, the fact that
$\forall x\in X,\ 1\leq a\left(x\right)<b\left(x\right)<+\infty$
and the results in Section \ref{subsec:Results-about-parametric},
the expressions
\begin{align*}
\sum_{m,n}\frac{\xi_{m,n}\left(s,x\right)\left(\frac{a\left(x\right)}{b\left(x\right)}\right)^{\frac{m}{d}}}{\left(\ell s+\eta+d-m+n\right)^{\mu+1-i}}, &  & \sum_{m,n}\frac{\xi_{m,n}\left(s,x\right)\left(\frac{a\left(x\right)}{b\left(x\right)}\right)^{\frac{n}{d}}}{\left(\ell s+\eta+d-m+n\right)^{\mu+1-i}}
\end{align*}
define functions in $\mathcal{A}\left(X\right)$ without poles outside
$P'$.

$\bullet$ If $\ell=0$ and $\eta\notin\mathbb{Z}$, then the above
equation holds for all $\left(s,x\right)\in\left(\mathbb{C}\setminus P\right)\times X$,
since the denominator does not vanish.

$\bullet$ If $\ell=0$ and $\eta\in\mathbb{Z}$, then we split $\Phi$
into the sum of two (still strongly convergent) series, by isolating
the indices which contribute, in $T$, to the power $y^{-1}$:
\begin{align*}
\Phi & \left(s,x,y\right)=\Phi_{=}\left(s,x,y\right)+\Phi_{\not=}\left(s,x,y\right)\\
 & =\sum_{\overset{m,n:}{m=\eta+d+n}}\xi_{m,n}\left(s,x\right)\left(\frac{a\left(x\right)}{y}\right)^{\frac{m}{d}}\left(\frac{y}{b\left(x\right)}\right)^{\frac{n}{d}}+\sum_{\overset{m,n:}{m\not=\eta+d+n}}\xi_{m,n}\left(s,x\right)\left(\frac{a\left(x\right)}{y}\right)^{\frac{m}{d}}\left(\frac{y}{b\left(x\right)}\right)^{\frac{n}{d}}\\
 & =y^{-\frac{\eta+d}{d}}\left(a\left(x\right)\right)^{\frac{\eta+d}{d}}\sum_{n}\xi_{n+\eta+d,n}\left(s,x\right)\left(\frac{a\left(x\right)}{b\left(x\right)}\right)^{\frac{n}{d}}+\sum_{\overset{m,n:}{m\not=\eta+d+n}}\xi_{m,n}\left(s,x\right)\left(\frac{a\left(x\right)}{y}\right)^{\frac{m}{d}}\left(\frac{y}{b\left(x\right)}\right)^{\frac{n}{d}}.
\end{align*}

The integral of $T_{\not=}\left(s,x,y\right):=G_{0}\left(s,x\right)y^{\frac{\eta}{d}}\left(\log y\right)^{\mu}\Phi_{\not=}\left(s,x,y\right)$
is computed as in the previous cases, and the denominators never vanish.

As for $T_{=}\left(s,x,y\right):=G_{0}\left(s,x\right)y^{\frac{\eta}{d}}\left(\log y\right)^{\mu}\Phi_{=}\left(s,x,y\right)$,
for $\left(s,x\right)\in\left(\mathbb{C}\setminus P\right)\times X$,
we have
\begin{align*}
\int_{a\left(x\right)}^{b\left(x\right)}T_{=}\left(s,x,y\right)\text{d}y & =G_{0}\left(s,x\right)\left(a\left(x\right)\right)^{\frac{\eta+d}{d}}\sum_{n}\xi_{n+\eta+d,n}\left(s,x\right)\left(\frac{a\left(x\right)}{b\left(x\right)}\right)^{\frac{n}{d}}\frac{\left(\log b\left(x\right)\right)^{\mu+1}}{\mu+1}\\
 & -G_{0}\left(s,x\right)\left(a\left(x\right)\right)^{\frac{\eta+d}{d}}\sum_{n}\xi_{n+\eta+d,n}\left(s,x\right)\left(\frac{a\left(x\right)}{b\left(x\right)}\right)^{\frac{n}{d}}\frac{\left(\log a\left(x\right)\right)^{\mu+1}}{\mu+1}.
\end{align*}

Hence we have shown that there is $H\in\mathcal{D}\left(X\right)$
without poles outside some closed discrete set $P'\supseteq P$, such
that
\[
\forall\left(s,x\right)\in\left(\mathbb{C}\setminus P'\right)\times X,\ H\left(s,x\right)=\int_{a\left(x\right)}^{b\left(x\right)}T\left(s,x,y\right)\text{d}y.
\]

\begin{rem}
\label{rem: holom continuation}If $\ell=0$ then $H$ has no new
singularities. If $\ell\not=0$, let $\sigma\in P'\setminus P$. Since
for all $\left(x,y\right)\in B,\text{ the function}\ s\longmapsto T\left(s,x,y\right)$
is holomorphic and bounded in a neighbourhood of $\sigma$, by differentiation
under the integral sign, the integral $\int_{a\left(x\right)}^{b\left(x\right)}T\left(s,x,y\right)\text{d}y$
is also holomorphic in a neighbourhood of $\sigma$. Since such an
integral coincides with $H$ on a deleted neighbourhood of $\sigma$
and $s\longmapsto H\left(s,x\right)$ is meromorphic, $\sigma$ is
not a pole of $H$ but a removable singularity. Hence,
\begin{align*}
H_{\sigma}\left(x\right) & :=\lim_{s\longrightarrow\sigma}H\left(s,x\right)=\lim_{s\longrightarrow\sigma}\int_{a\left(x\right)}^{b\left(x\right)}T\left(s,x,y\right)\text{d}y\\
 & =\int_{a\left(x\right)}^{b\left(x\right)}\lim_{s\longrightarrow\sigma}T\left(s,x,y\right)\text{d}y=\int_{a\left(x\right)}^{b\left(x\right)}T\left(\sigma,x,y\right)\text{d}y.
\end{align*}
The rightmost integral can be computed in a similar way as we did
above for the case $\ell=0,\eta\in\mathbb{Z}$ (where now we split
the series according to the condition $m=\ell\sigma+\eta+d+n$) and
the computation clearly shows that $H_{\sigma}\in\mathcal{C}^{\mathbb{K}}\left(X\right)$.

Finally, notice that every $\sigma\in P'\setminus P$ has the form
$\sigma=\frac{\nu_{0}-\eta-d}{\ell}$ for some $\nu_{0}\in\mathbb{Z}$,
so that if $\ell$ and/or $\eta$ are in $\mathbb{K}$, then so is
$\sigma$. 
\end{rem}
Hence, we have proven the following statement.
\begin{prop}
\label{prop: integration of C^M bdd generators}Let $B$ be as in
\eqref{eq:A_theta=00003DB} with $b<+\infty$, $\mathbb{K}\subseteq\mathbb{C}$
be a subfield and let $\mathcal{D}$ be either $\mathcal{C}^{\mathbb{K},\mathcal{M}}$
or $\mathcal{C}^{\mathcal{P}\left(\mathbb{K}\right),\mathcal{M}}$.
Let $T\in\mathcal{D}\left(B\right)$ be a prepared generator with
no poles outside $P$ (for some discrete closed set $P\subseteq\mathbb{K}$),
as in Definition \ref{def: prep gen of C^M}. Let
\[
P'=P\cup\left\{ s\in\mathbb{C}:\ \ell s+\eta\in\mathbb{\mathbb{Z}}\right\} \subseteq\mathbb{K}.
\]
Then
\[
\mathrm{Int}\left(T;\left(\mathbb{C}\setminus P\right)\times X\right)=\left(\mathbb{C}\setminus P\right)\times X
\]
and there exist a function $H\in\mathcal{D}\left(X\right)$ without
poles outside $P'$ such that
\[
\forall\left(s,x\right)\in\left(\mathbb{C}\setminus P'\right)\times X,\ H\left(s,x\right)=\int_{a\left(x\right)}^{b\left(x\right)}T\left(s,x,y\right)\mathrm{d}y.
\]
Moreover, for all $\sigma\in P'\setminus P$ there is a function $H_{\sigma}\in\mathcal{C}^{\mathbb{\mathbb{K}}}\left(X\right)$
such that
\[
\forall x\in X,\ H_{\sigma}\left(x\right)=\int_{a\left(x\right)}^{b\left(x\right)}T\left(\sigma,x,y\right)\mathrm{d}y
\]
and $\forall x\in X$, the function $s\longmapsto H\left(s,x\right)$
can be holomorphically extended at $s=\sigma$ by setting $H\left(\sigma,x\right)=H_{\sigma}\left(x\right)$.
\end{prop}
\begin{rem}
\label{rem: int of sum of gen on bdd}The proposition also applies
to any finite sum of prepared generators on the bounded cell $B$,
with $P'$ a finite union of closed and discrete sets and $P'\setminus P$
contained in a finitely generated $\mathbb{Z}$-lattice.
\end{rem}

\subsection{Cells with unbounded $y$-fibres\label{subsec:Unbounded-cells}}

We now introduce a type of function in $\mathcal{D}\left(X\times\mathbb{R}\right)$
which has a particularly simple expression in the last variable $y$. 
\begin{defn}
\label{def: C^M naive in y}Let $A\subseteq X\times\mathbb{R}$ be
a subanalytic cell which is open over $X$ (see Definition \ref{def:cellopen over X}).
A function $h\in\mathcal{D}\left(A\right)$ without poles outside
some closed discrete set $P\subseteq\mathbb{K}$ is \emph{Puiseux
in $y$} if there are $\ell,\eta\in\mathbb{K},\ d\in\mathbb{N}\setminus\left\{ 0\right\} ,\ \mu\in\mathbb{N}$
and a collection $\left\{ g_{k}\left(s,x\right)\right\} _{k\in\mathbb{N}}\subseteq\mathcal{D}\left(X\right)$
such that for all $s\in\mathbb{C}\setminus P$, the series of functions
\[
\varphi\left(s,x,y\right):=\sum_{k}g_{k}\left(s,x\right)y^{-\frac{k}{d}}
\]
converges normally on $A$ and $\forall\left(x,y\right)\in A,\ \mathbb{C}\setminus P\ni s\longmapsto\varphi\left(s,x,y\right)$
is holomorphic, and
\begin{equation}
h\left(s,x,y\right)=\varphi\left(s,x,y\right)y^{\frac{\ell s+\eta}{d}}\left(\log y\right)^{\mu}=\sum_{k}g_{k}\left(s,x\right)y^{\frac{\ell s+\eta-k}{d}}\left(\log y\right)^{\mu}.\label{eq:C^M naive in y}
\end{equation}
We call the tuple $\left(\ell,\eta,d,\mu\right)$ the \emph{Puiseux
data} of $h$.
\end{defn}
\begin{rem}
\label{rem: C naive gen}Let $B$ be as in \eqref{eq:A_theta=00003DB}
and $T\in\mathcal{D}\left(B\right)$ be a prepared generator (for
some $y$-prepared $1$-bounded subanalytic map $\psi$ as in \eqref{eq: psi}).
If $B$ has unbounded $y$-fibres, then $T$ is Puiseux in $y$. 
\end{rem}
We now turn our attention to prepared generators of $\mathcal{D}\left(B\right)$,
where, in the definition \eqref{eq:A_theta=00003DB} of $B$, we have
$b\equiv+\infty$. More generally, in what follows we will suppose
that $T\in\mathcal{D}\left(B\right)$ is a finite sum of prepared
generators (where $\psi$ is as in \eqref{eq: psi unbounded}), sharing
the same Puiseux data and without poles outside some closed discrete
set $P\subseteq\mathbb{K}$. Hence, for some $\ell,\eta\in\mathbb{K},\ \mu\in\mathbb{N}$,
$T$ has the form 

\begin{equation}
\begin{aligned}T\left(s,x,y\right) & =\sum_{j\leq N}T_{j}\left(s,x,y\right)\\
 & =\sum_{j\leq N}G_{j}\left(s,x\right)y^{\frac{\ell s+\eta}{d}}\left(\log y\right)^{\mu}\sum_{k}\xi_{j,k}\left(s,x\right)\left(\frac{a\left(x\right)}{y}\right)^{\frac{k}{d}}\\
 & =y^{\frac{\ell s+\eta}{d}}\left(\log y\right)^{\mu}\sum_{k}h_{k}\left(s,x\right)\left(\frac{a\left(x\right)}{y}\right)^{\frac{k}{d}},
\end{aligned}
\label{eq:prep unbounded naive}
\end{equation}
where $h_{k}=\sum_{j\leq N}G_{j}\xi_{j,k}\in\mathcal{D}\left(X\right)$. 

First, we describe $\mathrm{Int}\left(T;\left(\mathbb{C}\setminus P\right)\times X\right)$.
Let $m_{k}\left(s,y\right)=y^{\frac{\ell s+\eta-k}{d}}\left(\log y\right)^{\mu}$
and notice that, since $a\left(x\right)\geq1$ and since for all $s\in\mathbb{C}$
the real parts of the exponent of $y$ in $m_{k}$ and $m_{k'}$ are
different if $k\not=k'$,
\[
\mathrm{Int}\left(T;\left(\mathbb{C}\setminus P\right)\times X\right)=\bigcap_{k\in\mathbb{N}}\mathrm{Int}\left(h_{k}m_{k};\left(\mathbb{C}\setminus P\right)\times X\right).
\]

$\bullet$ If $\ell\not=0$ then
\begin{align*}
\mathrm{Int}\left(h_{k}m_{k};\left(\mathbb{C}\setminus P\right)\times X\right) & =\left\{ s\in\mathbb{C}\setminus P:\ \Re\left(\ell s+\eta\right)+d-k<0\right\} \times X\\
 & \cup\left\{ \left(s,x\right)\in\left(\mathbb{C}\setminus P\right)\times X:\ \Re\left(\ell s+\eta\right)+d-k\geq0\wedge h_{k}\left(s,x\right)=0\right\} 
\end{align*}
and hence, if
\begin{equation}
S_{0}=\left\{ s\in\mathbb{C}:\ \Re\left(\ell s+\eta\right)+d<0\right\} \ \text{and}\ S_{i}=\left\{ s\in\mathbb{C}:\ i-1\leq\Re\left(\ell s+\eta\right)+d<i\right\} \ \left(i\geq1\right),\label{eq:S_i}
\end{equation}
then
\begin{equation}
\mathrm{Int}\left(T;\left(\mathbb{C}\setminus P\right)\times X\right)=\left(S_{0}\times X\right)\cup\bigcup_{i\geq1}\left\{ \left(s,x\right)\in\left(S_{i}\setminus P\right)\times X:\ \bigwedge_{k<i}h_{k}\left(s,x\right)=0\right\} .\label{eq: int locus unbdd elln0}
\end{equation}

$\bullet$ If $\ell=0$ then 
\[
\mathrm{Int}\left(h_{k}m_{k};\left(\mathbb{C}\setminus P\right)\times X\right)=\begin{cases}
\left(\mathbb{C}\setminus P\right)\times X & \text{if}\ \Re\left(\eta\right)+d-k<0\\
\left\{ \left(s,x\right)\in\left(\mathbb{C}\setminus P\right)\times X:\ h_{k}\left(s,x\right)=0\right\}  & \text{if }\Re\left(\eta\right)+d-k\geq0
\end{cases}
\]
and hence, if $k_{0}=\lfloor\Re\left(\eta\right)\rfloor+d$, then
\begin{equation}
\mathrm{Int}\left(T;\left(\mathbb{C}\setminus P\right)\times X\right)=\left\{ \left(s,x\right)\in\left(\mathbb{C}\setminus P\right)\times X:\ \bigwedge_{k\leq k_{0}}h_{k}\left(s,x\right)=0\right\} .\label{eq: int locus unbdd ell=00003D0}
\end{equation}
Let 
\begin{equation}
P'=\begin{cases}
P\cup\left\{ s\in\mathbb{C}:\ \Re\left(\ell s+\eta\right)+d\in\mathbb{N}\right\}  & \text{if }\ell\not=0\\
P & \text{if }\ell=0
\end{cases}\subseteq\mathbb{K}.\label{eq:new poles}
\end{equation}
Notice that $\left(P'\setminus P\right)\cap S_{0}=\emptyset$.

Our next aim is to show that there exists $H\in\mathcal{D}\left(X\right)$,
with no poles outside $P'$ such that $H$ coincides with the integral
of $T$ on its integration locus.

In the notation of Lemma \ref{lem: integral of a naive in y}, let
\begin{align*}
H_{k}\left(s,x\right) & =-\left(a\left(x\right)\right)^{\frac{\ell s+\eta+d}{d}}\sum_{i\leq\mu}c_{\mu,i}\left(\log a\left(x\right)\right)^{i}\frac{h_{k}\left(s,x\right)}{\left(\ell s+\eta+d-k\right)^{\mu+1-i}},
\end{align*}
and define
\[
H\left(s,x\right)=\begin{cases}
{\displaystyle \sum_{k\geq0}H_{k}}\left(s,x\right) & \text{if }\ell\not=0\\
\\
{\displaystyle \sum_{k>k_{0}}}H_{k}\left(s,x\right) & \text{if }\ell=0
\end{cases}.
\]
By the results in Section \ref{subsec:Results-about-parametric},
$H\in\mathcal{D}\left(X\right)$ and has no poles outside $P'$, and
by Lemma \ref{lem: integral of a naive in y},

\[
\forall\left(s,x\right)\in\text{Int}\left(T;\left(\mathbb{C}\setminus P'\right)\times X\right),\ \int_{a\left(x\right)}^{+\infty}T\left(s,x,y\right)\text{d}y=H\left(s,x\right).
\]
If $\ell=0$ then $H$ has no new singularities, whereas if $\ell\not=0$
then the new singularities are located in $\left(\mathbb{C}\setminus S_{0}\right)\times X$
and are in general not removable. 
\begin{rem}
\label{rem: int locus of unbdd gen}If $\mathcal{D}=\mathcal{C}^{\mathbb{K},\mathcal{M}}$
then the sets $S_{i}\ \left(i\geq1\right)$ in \eqref{eq:S_i} are
vertical strips in the complex plane of fixed width $\frac{1}{\ell}$.
The points $\sigma\in P'\setminus P$ lie on the boundaries of such
strips and their imaginary part is equal to ${\displaystyle \frac{\Im\left(\eta\right)}{\ell}}$.
If $\mathcal{D}=\mathcal{C}^{\mathcal{P}\left(\mathbb{K}\right),\mathcal{M}}$,
where $\mathbb{K}\not\subseteq\mathbb{R}$, then $\ell\in\mathbb{K}$
and the sets $S_{i}$ are parallel (not necessarily vertical) strips
of fixed width. The points $\sigma\in P'\setminus P$ again lie on
the boundaries of such strips and satisfy the equation $\Re\left(\ell\right)\Im\left(\sigma\right)+\Im\left(\ell\right)\Re\left(\sigma\right)+\Im\left(\eta\right)=0$.
In both cases, the set $P'\setminus P$ is contained in a finitely
generated $\mathbb{Z}$-lattice and hence $P'$ is closed and discrete.
\end{rem}
Hence, we have proven the following result.
\begin{prop}
\label{prop: integration of C^M unbdd generators}Let $B$ be as in
\eqref{eq:A_theta=00003DB} with $b=+\infty$, $\mathbb{K}\subseteq\mathbb{C}$
be a subfield and let $\mathcal{D}$ be either $\mathcal{C}^{\mathbb{K},\mathcal{M}}$
or $\mathcal{C}^{\mathcal{P}\left(\mathbb{K}\right),\mathcal{M}}$.
Let $T\in\mathcal{D}\left(B\right)$ be a finite sum of prepared generators
sharing the same Puiseux data, as in \eqref{eq:prep unbounded naive},
with no poles outside $P$ (for some discrete closed set $P\subseteq\mathbb{K}$).
Then $\mathrm{Int}\left(T;\left(\mathbb{C}\setminus P\right)\times X\right)$
is described as in \eqref{eq: int locus unbdd elln0} (if $\ell\not=0$)
or in \eqref{eq: int locus unbdd ell=00003D0} (if $\ell=0$) and,
for $P'$ as in \eqref{eq:new poles}, there exists a function $H\in\mathcal{D}\left(X\right)$
without poles outside $P'$ such that
\[
\forall\left(s,x\right)\in\mathrm{Int}\left(T;\left(\mathbb{C}\setminus P\right)\times X\right),\ H\left(s,x\right)=\int_{a\left(x\right)}^{+\infty}T\left(s,x,y\right)\mathrm{d}y.
\]
\end{prop}

\section{Stability under integration of (parametric) power-constructible functions\label{sec:Stability-of (parametric) powers}}

This section is devoted to the proof of the results of stability under
parametric integration in Section \ref{sec:Main results}.

For the rest of this section we let $\mathcal{D}$ be either $\mathcal{C}^{\mathbb{K},\mathcal{M}}$
or $\mathcal{C}^{\mathcal{P}\left(\mathbb{K}\right),\mathcal{M}}$.

We will first prove stability under integration when we integrate
with respect to a single variable $y$. In this case, we can also
give a description of the integration locus. The strategy is the following:
we prepare the function we want to integrate with respect to the variable
$y$. This produces a cell decomposition such that on each cell, in
the new coordinates the function is a sum of prepared generators.
If the cell has bounded $y$-fibres, then the function is integrable
everywhere in restriction to such a cell, and we have already shown
(see Remark \ref{rem: int of sum of gen on bdd}) that the integral
can be expressed as a function of $\mathcal{D}$. If the cell has
unbounded $y$-fibres and the prepared generators all share the same
Puiseux data, then we know how to conclude by the results of the previous
section. It remains to consider the case of a sum of generators who
have different Puiseux data. Such data induce a partition of $\mathbb{C}$
into areas (see \eqref{eq:S_i}) which are involved in the description
of the integrability locus. In order to deal with different Puiseux
data, we introduce the notion of non-accumulating grid.
\begin{defn}
\label{def: grid and partition}Given $N,d\in\mathbb{N}^{\times}$
and $\left\{ \left(\ell_{i},\eta_{i}\right):\ 0\leq i\leq N\right\} \subseteq\mathbb{K}^{2}$,
define
\begin{align*}
\Xi_{i,0,-} & =\emptyset,\\
\Xi_{i,0,\circ} & =\left\{ s\in\mathbb{C}:\ \Re\left(\ell_{i}s+\eta_{i}\right)+d<0\right\} \quad\left(i\leq N\right),\\
\Xi_{i,j,-} & =\left\{ s\in\mathbb{C}:\ \Re\left(\ell_{i}s+\eta_{i}\right)+d=j-1\right\} \quad\left(i\leq N,\ j\in\mathbb{N}^{\times}\right),\\
\Xi_{i,j,\circ} & =\left\{ s\in\mathbb{C}:\ j-1<\Re\left(\ell_{i}s+\eta_{i}\right)+d<j\right\} \quad\left(i\leq N,\ j\in\mathbb{N}^{\times}\right).
\end{align*}
A collection of sets (partitioning $\mathbb{C}$) of the form
\[
\mathcal{G}=\left\{ \Xi_{i,j,\star}:\ i\leq N,\ j\in\mathbb{N},\ \star\in\left\{ -,\circ\right\} \right\} 
\]
is called a\emph{ non-accumulating grid} of \emph{data} $\left\{ N,d,\left(\ell_{0},\eta_{0}\right),\ldots,\left(\ell_{N},\eta_{N}\right)\right\} $.
Note that if $\ell_{i}=0$ then $\forall j\in\mathbb{N},\forall\star\in\left\{ -,\circ\right\} ,\ \Xi_{i,j,\star}$
is either empty or the whole $\mathbb{C}$.

A \emph{$\mathcal{G}$-cell} is a nonempty subset $\Sigma\subseteq\mathbb{C}$
such that
\[
\forall\Xi\in\mathcal{G},\ \Xi\cap\Sigma=\emptyset\text{\ or }\Sigma\subseteq\Xi\text{, and }\Sigma=\bigcap\left\{ \Xi\in\mathcal{G}:\ \Sigma\subseteq\Xi\right\} .
\]
We let $\mathcal{R}\left(\mathcal{G}\right)$ be the collection of
all $\mathcal{G}$-cells. The $\mathcal{G}$-cells are convex and
form a partition of $\mathbb{C}$. Each $\mathcal{G}$-cell either
has empty interior (an isolated point, a segment or a line) or is
an open subset of $\mathbb{C}$ containing an open ball of radius
$\varepsilon$, for some $\varepsilon=\varepsilon\left(\mathcal{G}\right)>0$
depending only on $\mathcal{G}$ (hence the word \textquotedblleft non-accumulating\textquotedblright ).
Given a $\mathcal{G}$-cell $\Sigma$, there are functions $j_{\Sigma}:\left\{ 0,\ldots,N\right\} \longrightarrow\mathbb{N}$
and $\star_{\Sigma}:\left\{ 0,\ldots,N\right\} \longrightarrow\left\{ -,\circ\right\} $
such that $\Sigma=\bigcap_{i\leq N}\Xi_{i,j_{\Sigma}\left(i\right),\star_{\Sigma}\left(i\right)}$.

If all the $\ell_{i}$ are in $\mathbb{R}^{\times}$, the we say that
$\mathcal{G}$ is a \emph{vertical} non-accumulating grid. In this
case, the cells with empty interior are points or vertical lines,
and the open cells are vertical strips of width $\geq\varepsilon$,
for some $\varepsilon=\varepsilon\left(\mathcal{G}\right)>0$. 
\end{defn}
\begin{example}
\label{ex: grid for one gen}Let $N,d\in\mathbb{N}^{\times}$. For
$i\leq N$, let $T_{i}$ be a sum of prepared generators on an unbounded
cell, sharing the same Puiseux data $\left(\ell_{i},\eta_{i},d,\mu_{i}\right)$
(as in \eqref{eq:prep unbounded naive}, see Remark \ref{rem: C naive gen}),
without poles outside some closed discrete set $P\subseteq\mathbb{K}$.
Consider the non-accumulating grid of data $\left\{ N,d,\left(\ell_{0},\eta_{0}\right),\ldots,\left(\ell_{N},\eta_{N}\right)\right\} $
and let $\Sigma=\bigcap_{i\leq N}\Xi_{i,j_{\Sigma}\left(i\right),\star_{\Sigma}\left(i\right)}\in\mathcal{R}\left(\mathcal{G}\right)$
be a $\mathcal{G}$-cell. Then
\[
\text{Int}\left(T_{i};\left(\Sigma\setminus P\right)\times X\right)=\left\{ \left(s,x\right):\ s\in\Sigma\setminus P,\ \bigwedge_{k<j_{\Sigma}\left(i\right)}g_{i,k}\left(s,x\right)=0\right\} ,
\]
where $g_{i,k}\in\mathcal{D}\left(X\right)$ are the coefficients
in the series expansion \eqref{eq:C^M naive in y} of $T_{i}$. It
follows that, if we rename 
\[
\left\{ g_{k}^{\Sigma}:\ k\in J_{\Sigma}\right\} =\left\{ g_{i,k}:\ i\leq N,\ k<j_{\Sigma}\left(i\right)\right\} ,
\]
then
\[
\bigcap_{i\leq N}\text{Int}\left(T_{i};\left(\Sigma\setminus P\right)\times X\right)=\left\{ \left(s,x\right):\ s\in\Sigma\setminus P,\ \bigwedge_{k\in J_{\Sigma}}g_{k}^{\Sigma}\left(s,x\right)=0\right\} .
\]
\end{example}
\begin{thm}
\label{thm: interpolation and locus C^M}Let $\mathbb{K}\subseteq\mathbb{C}$
be a subfield and let $\mathcal{D}$ be either $\mathcal{C}^{\mathbb{K},\mathcal{M}}$
or $\mathcal{C}^{\mathcal{P}\left(\mathbb{K}\right),\mathcal{M}}$.
Let $P\subseteq\mathbb{K}$ be a closed discrete set and $h\in\mathcal{D}\left(X\times\mathbb{R}\right)$
be with no poles outside $P$. There exist a closed discrete set $P'\subseteq\mathbb{K}$,
containing $P$ and contained in a finitely generated $\mathbb{Z}$-lattice,
and a function $H\in\mathcal{D}\left(X\right)$ without poles outside
$P'$ such that
\[
\forall\left(s,x\right)\in\mathrm{Int}\left(h;\left(\mathbb{C}\setminus P'\right)\times X\right),\quad\int_{\mathbb{R}}h\left(s,x,y\right)\mathrm{d}y=H\left(s,x\right).
\]
Moreover, there exists a non-accumulating grid $\mathcal{G}$ as in
Definition \ref{def: grid and partition} such that
\begin{equation}
\mathrm{Int}\left(h;\left(\mathbb{C}\setminus P'\right)\times X\right)=\bigcup_{\Sigma\in\mathcal{P\left(\mathcal{G}\right)}}\left\{ \left(s,x\right):\ s\in\Sigma\setminus P',\ \bigwedge_{k\in J_{\Sigma}}g_{k}^{\Sigma}\left(s,x\right)=0\right\} ,\label{eq:int locus}
\end{equation}
for a suitable finite set $J_{\Sigma}$ and suitable $g_{k}^{\Sigma}\in\mathcal{D}\left(X\right)$
without poles outside $P$.
\end{thm}
\begin{proof}
Apply Proposition \ref{prop: prep of C^M} to $h$ to find a cell
decomposition of $\mathbb{R}^{m+1}$ such that on each cell $B_{A}$
as in \eqref{eq: Atheta}, $h\circ\Pi_{A}$ is a finite sum of prepared
generators (for some $y$-prepared $1$-bounded subanalytic map $\psi_{A}$).
We may suppose that $X$ itself is a cell and we concentrate on the
collection $\mathcal{X}$ of all the cells of the decomposition which
have $X$ as a base, and which are open over $\mathbb{R}^{m}$. Since
$\text{Int}\left(h;\left(\mathbb{C}\setminus P\right)\times X\right)=\bigcap_{A\in\mathcal{X}}\text{Int}\left(h\cdot\chi_{A};\left(\mathbb{C}\setminus P\right)\times X\right)$
and
\[
\forall\left(s,x\right)\in\text{Int}\left(h;\left(\mathbb{C}\setminus P\right)\times X\right),\ \int_{\mathbb{R}}h\left(s,x,y\right)\text{d}y=\sum_{A\in\mathcal{X}}\int_{\mathbb{R}}h\left(s,x,y\right)\cdot\chi_{A}\left(x,y\right)\text{d}y,
\]
it is enough to prove the theorem for the functions $h\cdot\chi_{A}$.

For $A\in\mathcal{X}$, we can write
\[
h\circ\Pi_{A}\left(s,x,y\right)=\sum_{i\leq M_{A}}\tilde{T}_{i}^{A}\left(s,x,y\right),
\]
where each $\tilde{T}_{i}^{A}\in\mathcal{D}\left(B_{A}\right)$ is
a prepared generator. Recall the notation in \eqref{eq:A} and note
that 
\[
\frac{\partial\Pi_{A}}{\partial y}\left(x,y\right)=\sigma_{A}\tau_{A}y^{\tau_{A}-1}.
\]
Define
\[
T_{i}^{A}\left(s,x,y\right):=\sigma_{A}\tau_{A}y^{\tau_{A}-1}\tilde{T}_{i}^{A}\left(s,x,y\right).
\]
Then, 
\[
\text{Int}\left(\tilde{T}_{i}^{A}\circ\Pi_{A}^{-1};\left(\mathbb{C}\setminus P\right)\times X\right)=\text{Int}\left(T_{i}^{A};\left(\mathbb{C}\setminus P\right)\times X\right)
\]
and $\forall\left(s,x\right)\in\mathrm{Int}\left(h\cdot\chi_{A};\left(\mathbb{C}\setminus P\right)\times X\right),$
\begin{align*}
\int_{\mathbb{R}}h\left(s,x,y\right)\cdot\chi_{A}\left(x,y\right)\text{d}y & =\int_{a_{A}\left(x\right)}^{b_{A}\left(x\right)}h\circ\Pi_{A}\left(s,x,y\right)\cdot\frac{\partial\Pi_{A}}{\partial y}\left(x,y\right)\text{d}y\\
 & =\int_{a_{A}\left(x\right)}^{b_{A}\left(x\right)}\sum_{i\leq M_{A}}T_{i}^{A}\left(s,x,y\right)\text{d}y.
\end{align*}

If $B_{A}$ has bounded $y$-fibres, then by Proposition \ref{prop: integration of C^M bdd generators}
and Remark \ref{rem: int of sum of gen on bdd},
\[
\text{Int}\left(T_{i}^{A};\left(\mathbb{C}\setminus P\right)\times X\right)=\left(\mathbb{C}\setminus P\right)\times X
\]
and there are a closed discrete set $P'_{A}\subseteq\mathbb{K}$ (containing
$P$ and contained in a finitely generated $\mathbb{Z}$-lattice)
and functions $H_{i}^{A}\in\mathcal{D}\left(X\right)$ without poles
outside $P'_{A}$, such that 
\[
\forall\left(s,x\right)\in\left(\mathbb{C}\setminus P_{A}'\right)\times X,\ \sum_{i\leq M_{A}}H_{i}^{A}\left(s,x\right)=\int_{\mathbb{R}}h\left(s,x,y\right)\cdot\chi_{A}\left(s,x\right)\text{d}y.
\]

If $B_{A}$ has unbounded $y$-fibres, then consider the prepared
generators $\tilde{T}_{i}^{A}$ (which are Puiseux in $y$, of Puiseux
data $\left(\ell_{i}',\eta_{i}',d,\mu_{i}'\right)$). Suppose that
there are $i\not=j\leq M_{A}$ such that $\ell_{i}'=\ell_{j}',\ \mu_{i}'=\mu_{j}'$
and $\eta_{i}'-\eta_{j}'=\nu\in\mathbb{N}$. Write
\begin{align*}
\tilde{T}_{j}^{A}\left(s,x,y\right) & =\sum_{k}\tilde{g}_{j,k}\left(s,x\right)y^{\frac{\ell_{j}'s+\eta_{j}'-k}{d}}\left(\log y\right)^{\mu_{j}'}\\
 & =\sum_{k}h_{j,k}\left(s,x\right)y^{\frac{\ell_{i}'s+\eta_{i}'-k}{d}}\left(\log y\right)^{\mu_{i}'},
\end{align*}
where
\[
h_{j,k}\left(s,x\right)=\begin{cases}
0 & \text{if }k<\nu\\
\tilde{g}_{j,k-\nu} & \text{if }k\geq\nu
\end{cases}.
\]
Now $\tilde{T}_{i}^{A}$ and $\tilde{T}_{j}^{A}$ share the same Puiseux
data (and so do $T_{i}^{A}$ ans $T_{j}^{A}$). Hence, by summing
together all generators which share the same Puiseux data, we may
write
\[
\sum_{i\leq M_{A}}\tilde{T}_{i}^{A}\left(s,x,y\right)=\sum_{i\leq N_{A}}\tilde{T}_{i}\left(s,x,y\right),
\]
where $N_{A}\in\mathbb{N}$ and, if $T_{i}=\sigma_{A}\tau_{A}y^{\tau_{A}-1}\tilde{T}_{i}$,
\begin{equation}
T_{i}\left(s,x,y\right)=\sum_{k}g_{i,k}\left(s,x\right)y^{\frac{\ell_{i}s+\eta_{i}-k}{d}}\left(\log y\right)^{\mu_{i}}\in\mathcal{D}\left(B_{A}\right)\label{eq: sum of gen with different puiseux data}
\end{equation}
is a finite sum of prepared generators on the unbounded cell $B_{A}$
sharing the same Puiseux data $\left(\ell_{i},\eta_{i},d,\mu_{i}\right)$.
Moreover, $\forall i\not=j\leq N_{A},\ \left(\ell_{i},\eta_{i},\mu_{i}\right)\not=\left(\ell_{j},\eta_{j},\mu_{j}\right)$
and if $\left(\ell_{i},\mu_{i}\right)=\left(\ell_{j},\mu_{j}\right)$
then $\eta_{i}-\eta_{j}\notin\mathbb{Z}$. Let
\[
P'_{A}=P\cup\left\{ s\in\mathbb{C}:\ \exists i\leq N_{A}\text{ s.t. }\ell_{i}\not=0\text{ and }\ell_{i}s+\eta_{i}+d\in\mathbb{N}\right\} .
\]
Apply Proposition \ref{prop: integration of C^M unbdd generators}
to each $T_{i}$ and find $H_{i}\in\mathcal{D}\left(X\right)$ without
poles outside $P'_{A}$ such that 
\[
\forall\left(s,x\right)\in\text{Int}\left(T_{i};\left(\mathbb{C}\setminus P'_{A}\right)\times X\right),\ H_{i}\left(s,x\right)=\int_{a_{A}\left(x\right)}^{+\infty}T_{i}\left(s,x,y\right)\text{d}y.
\]
Clearly, $\bigcap_{i\leq N_{A}}\mathrm{Int}\left(T_{i};\left(\mathbb{C}\setminus P\right)\times X\right)\subseteq\mathrm{Int}\left(h\cdot\chi_{A};\left(\mathbb{C}\setminus P\right)\times X\right)$
and
\[
\forall\left(s,x\right)\in\bigcap_{i}\mathrm{Int}\left(T_{i};\left(\mathbb{C}\setminus P_{A}'\right)\times X\right),\ \int_{\mathbb{R}}h\left(s,x,y\right)\cdot\chi_{A}\left(x,y\right)\text{d}y=H_{0}+\cdots+H_{N}.
\]
Recall that the description of the above integrability locus is given
in Example \ref{ex: grid for one gen}, with respect to the non-accumulating
grid $\mathcal{G}_{A}$ of data $\left\{ N_{A},d,\left(\ell_{0},\eta_{0}\right),\ldots,\left(\ell_{N_{A}},\eta_{N_{A}}\right)\right\} $.
We would hence be done if we could show that the integrability locus
of $h\cdot\chi_{A}$ coincided with the intersection of the integrability
loci of the $T_{i}$. This is the case, outside a closed discrete
set, as we now show.

Let
\begin{align}
P_{A}'' & =\left\{ s\in\mathbb{C}:\ \exists i\not=j\leq N_{A}\text{ s.t. }\mu_{i}=\mu_{j},\ \ell_{i}\not=\ell_{j}\text{ and }\left(\ell_{i}-\ell_{j}\right)s+\left(\eta_{i}-\eta_{j}\right)\in\mathbb{Z}\right\} \label{eq:collision set}
\end{align}
and notice that $P''_{A}\subseteq\mathbb{K}$ is contained in a finitely
generated $\mathbb{Z}$-lattice. Note that $\forall s\in\mathbb{C}\setminus P_{A}''$,
the tuples
\[
\left(\frac{\ell_{i}s+\eta_{i}-k}{d},\mu_{i}\right)\ \ \ \ 1\leq i\leq N_{A},\ k\in\mathbb{N}
\]
are pairwise distinct. 

We now show that $\mathrm{Int}\left(h\cdot\chi_{A};\left(\mathbb{C}\setminus P_{A}''\right)\times X\right)=\bigcap_{i}\mathrm{Int}\left(T_{i};\left(\mathbb{C}\setminus P_{A}''\right)\times X\right)$. 

Let $\Sigma=\bigcap_{i\leq N}\Xi_{i,j_{\Sigma}\left(i\right),\star_{\Sigma}\left(i\right)}$
be a $\mathcal{G}_{A}$-cell, in the notation of Example \ref{ex: grid for one gen},
and let $\left(s_{0},x_{0}\right)\in\mathrm{Int}\left(h\cdot\chi_{A};\left(\Sigma\setminus P''_{A}\right)\times X\right)$.
For all $\left(s,x,y\right)\in\left(\Sigma\setminus P''_{A}\right)\times B_{A}$,
write
\begin{align*}
\sum_{i=1}^{N_{A}}T_{i}\left(s,x,y\right) & =\left(\sum_{i=1}^{N_{A}}\sum_{k=0}^{j_{\Sigma}\left(i\right)-1}g_{i,k}\left(s,x\right)y^{\frac{\ell_{i}s+\eta_{i}-k}{d}}\left(\log y\right)^{\mu_{i}}\right)+\left(\sum_{i=1}^{N_{A}}\sum_{k\geq j_{\Sigma}\left(i\right)}g_{i,k}\left(s,x\right)y^{\frac{\ell_{i}s+\eta_{i}-k}{d}}\left(\log y\right)^{\mu_{i}}\right)\\
 & =h_{A,1}^{\Sigma}\left(s,x,y\right)+h_{A,2}^{\Sigma}\left(s,x,y\right)
\end{align*}
and notice that $\text{Int}\left(h_{A,2}^{\Sigma};\Sigma\times X\right)=\Sigma\times X$,
so $\left(s_{0},x_{0}\right)\in\text{Int}\left(h_{A,1}^{\Sigma};\left(\Sigma\setminus P''_{A}\right)\times X\right)$.
Rename the (finitely many) terms appearing in the double sum defining
$h_{A,1}^{\Sigma}$ as 
\[
\left\{ g_{j}^{\Sigma}\left(s,x\right)y^{\alpha_{j}s+\beta_{j}}\left(\log y\right)^{\nu_{j}}\right\} _{j\in J_{\Sigma}}
\]
and let
\[
a_{j}=\Re\left(\alpha_{j}s_{0}+\beta_{j}\right),\ b_{j}=\Im\left(\alpha_{j}s_{0}+\beta_{j}\right).
\]
Recall that $\left(a_{j},b_{j}\right)\not=\left(a_{j'},b_{j'}\right)$
whenever $\nu_{j}=\nu_{j'}$, since $s_{0}\notin P_{A}''$. Let $\left(a_{0},\nu_{0}\right)$
be the lexicographic maximum of the set $\left\{ \left(a_{j},\nu_{j}\right):\ j\in J_{\Sigma}\right\} $
and let $J_{0}=\left\{ j\in J_{\Sigma}:\ \left(a_{j},\nu_{j}\right)=\left(a_{0},\nu_{0}\right)\right\} $.
Write
\[
h_{A,1}^{\Sigma}\left(s_{0},x_{0},y\right)=y^{a_{0}}\left(\log y\right)^{\nu_{0}}\sum_{j\in J_{0}}g_{j}^{\Sigma}\left(s_{0},x_{0}\right)y^{\mathrm{i}b_{j}}+\sum_{j\in J_{\Sigma}\setminus J_{0}}g_{j}^{\Sigma}\left(s_{0},x_{0}\right)y^{a_{j}+\mathrm{i}b_{j}}\left(\log y\right)^{\nu_{j}}.
\]
Since $\left(s_{0},x_{0}\right)\in\text{Int}\left(h_{A,1}^{\Sigma};\left(\Sigma\setminus P''_{A}\right)\times X\right)$,
it follows from Proposition \ref{prop: non compensation powers} (in
the case where all the polynomials $p_{j}$ are identically zero)
that $\bigwedge_{j\in J_{0}}g_{j}^{\Sigma}\left(s_{0},x_{0}\right)=0.$
By repeating this procedure with the index set $J_{\Sigma}\setminus J_{0}$,
we end up obtaining that 
\[
\bigwedge_{j\in J_{\Sigma}}g_{j}^{\Sigma}\left(s_{0},x_{0}\right)=0,
\]
i.e. $\left(s_{0},x_{0}\right)\in\bigcap_{i\leq N_{A}}\mathrm{Int}\left(T_{i};\left(\Sigma\setminus P''_{A}\right)\times X\right)$. 

Summing up, if we define $P''=\bigcup\left\{ P''_{A}:\ B_{A}\ \text{unbounded}\right\} $,
$\mathcal{G}:=\bigcup\left\{ \mathcal{G}_{A}:\ B_{A}\ \text{unbounded}\right\} $
and $P':=\bigcup_{A\in\mathcal{X}}P'_{A}\cup P''$, then the proof
of the theorem is complete.
\end{proof}
\begin{rem}
\label{rem: addenda to stability C^M}In the previous proof, if $\sigma\in P_{A}''$,
then we rewrite the functions $T_{i}\left(\sigma,x,y\right)$ by regrouping
the terms with the same exponents. We obtain thus new functions $T_{i,\sigma}\in\mathcal{C}^{\mathbb{C}}\left(X\times\mathbb{R}\right)$
(seen as functions in $\mathcal{C}^{\mathcal{M}}\left(X\times\mathbb{R}\right)$
which happen not to depend on $s$) to which Proposition \ref{prop: integration of C^M unbdd generators}
applies and such that, if $h_{\sigma}\left(x,y\right)=h\left(\sigma,x,y\right)\cdot\chi_{A}\left(x,y\right)$,
then 
\[
\text{Int}\left(h_{\sigma};X\right)=\bigcap_{i}\text{Int}\left(T_{i,\sigma};X\right).
\]
Moreover, if $\sigma\in P_{A}''\setminus P'_{A}$ then $\sigma$ is
not a singularity of either of the $H_{i}$ and, since the computation
of the integral is done integrating term-by-term, it is still the
case that
\[
\int_{\mathbb{R}}h\left(\sigma,x,y\right)\cdot\chi_{A}\left(x,y\right)\text{d}y=H_{0}\left(\sigma,x\right)+\cdots+H_{N}\left(\sigma,x\right).
\]
\end{rem}
\begin{rem}
\label{rem: non vertical grid}The non-accumulating grid $\mathcal{G}$
in Theorem \ref{thm: interpolation and locus C^M} is vertical in
all but the case $\mathcal{D}=\mathcal{C}^{\mathcal{P}\left(\mathbb{K}\right),\mathcal{M}}$,
with $\mathbb{K}\not\subseteq\mathbb{R}$. This implies in particular
that the system $\mathcal{C}^{\mathcal{P}\left(\mathbb{C}\right),\mathcal{M}}$
is strictly larger that the system $\mathcal{C}^{\mathcal{M}}$: for
example, if $h\in\mathcal{C}^{\mathcal{P}\left(\mathbb{C}\right),\mathcal{M}}\left(X\times\mathbb{R}\right)$
is a finite sum of generators which are Puiseux in $y$ on some cell
$A$ with unbounded $y$-fibres (see Definition \ref{def: C^M naive in y}),
where the real and imaginary parts of the exponents $\ell$ appearing
in the Puiseux data are all nonzero, then the integration locus of
$h$ in \eqref{eq:int locus} is based on a non-accumulating grid
which is not vertical. Hence $h$ cannot be an element of $\mathcal{C}^{\mathcal{M}}$.
\end{rem}
We now conclude the proof of Theorem \ref{thm:variants}%
, using Fubini's Theorem.
\begin{proof}
We argue by induction on $n\in\mathbb{N}^{\times}$. If $n=1$ then
it is Theorem \ref{thm: interpolation and locus C^M}. We prove the
case $n+1$: let $y$ be an $n$-tuple of variables and let $z$ be
a single variable, and consider $h\in\mathcal{D}\left(X\times\mathbb{R}^{n+1}\right)$
without poles outside some closed discrete set $P$. By Fubini's Theorem,
for all $\left(s,x\right)\in\text{Int}\left(h;\left(\mathbb{C}\setminus P\right)\times X\right)$,
the set 
\[
E_{\left(s,x\right)}:=\left\{ y\in\mathbb{R}^{n}:\ \left(s,x,y\right)\in\text{Int}\left(h;\left(\mathbb{C}\setminus P\right)\times X\times\mathbb{R}^{n}\right)\right\} 
\]
is such that $\mathbb{R}^{n}\setminus E_{\left(s,x\right)}$ has measure
zero and
\[
\iint_{\mathbb{R}^{n+1}}h\left(s,x,y,z\right)\text{d}y\wedge\text{d}z=\int_{E_{\left(s,x\right)}}\left[\int_{\mathbb{R}}h\left(s,x,y,z\right)\text{d}z\right]\text{d}y.
\]
By Theorem \ref{thm: interpolation and locus C^M}, applied to $h$
as an element of $\mathcal{D}\left(\left(X\times\mathbb{R}^{n}\right)\times\mathbb{R}\right)$,
there exist a set $P_{1}\subseteq\mathbb{K}$ (containing $P$ and
contained in a finitely generated $\mathbb{Z}$-lattice) and a function
$H_{1}\in\mathcal{D}\left(X\times\mathbb{R}^{n}\right)$ without poles
outside $P_{1}$ such that
\[
\forall\left(s,x,y\right)\in\text{Int}\left(h;\left(\mathbb{C}\setminus P_{1}\right)\times X\times\mathbb{R}^{n}\right),\ H_{1}\left(s,x,y\right)=\int_{\mathbb{R}}h\left(s,x,y,z\right)\text{d}z.
\]
We now apply the inductive hypothesis to $H_{1}$ and find that there
exist $P'\subseteq\mathbb{K}$ (containing $P_{1}$ and contained
in a finitely generated $\mathbb{Z}$-lattice) and a function $H\in\mathcal{D}\left(X\right)$
without poles outside $P'$ such that 
\[
\forall\left(s,x\right)\in\text{Int}\left(H_{1};\left(\mathbb{C}\setminus P'\right)\times X\right),\ H\left(s,x\right)=\int_{\mathbb{R}^{n}}H_{1}\left(s,x,y\right)\text{d}y.
\]
Let $\left(s,x\right)\in\text{Int}\left(h;\left(\mathbb{C}\setminus P'\right)\times X\right)$.
Since $H_{1}$ is defined on the whole $\left(\mathbb{C}\setminus P'\right)\times X\times\mathbb{R}^{n}$
and $\mathbb{R}^{n}\setminus E_{\left(s,x\right)}$ has measure zero,
\[
\iint_{\mathbb{R}^{n+1}}h\left(s,x,y,z\right)\text{d}y\wedge\text{d}z=\int_{\mathbb{R}^{n}}H_{1}\left(s,x,y\right)\text{d}y.
\]
In particular, $\left(s,x\right)\in\text{Int}\left(H_{1};\left(\mathbb{C}\setminus P'\right)\times X\right)$
and 
\[
\iint_{\mathbb{R}^{n+1}}h\left(s,x,y,z\right)\text{d}y\wedge\text{d}z=H\left(s,x\right).
\]
\end{proof}
\begin{rem}
\label{rem: proof of 2.5}The proof of Theorem \ref{thm Stability of C^K}
is obtained as a special case of that of Theorem \ref{thm:variants},
where all the functions involved happen not to depend on the variable
$s$.
\end{rem}
\medskip{}

We conclude this section with some further remarks about the classes
$\mathcal{C}^{\mathbb{K}},\mathcal{C}^{\mathbb{K},\mathcal{M}},\mathcal{C}^{\mathcal{P}\left(\mathbb{K}\right),\mathcal{M}}$
considered here. Again, we let $\mathcal{D}$ be either $\mathcal{C}^{\mathbb{K},\mathcal{M}}$
or $\mathcal{C}^{\mathcal{P}\left(\mathbb{K}\right),\mathcal{M}}$.
\begin{rems}
\label{rems: closure properties}$\ $
\begin{enumerate}
\item Let $\Sigma\subseteq\mathbb{C}$ be open and define $\mathcal{D}_{\Sigma}\left(X\right):=\left\{ h\restriction\Sigma\times X:\ h\in\mathcal{D}\left(X\right)\right\} $.
Clearly, Theorem \ref{thm:variants} also holds for $\mathcal{D}_{\Sigma}$.
\item $\mathcal{D}$ is stable under right-composition with meromorphic
functions, in the following sense. Let $\xi\in\mathcal{E}_{\mathbb{K}}$
and $\Sigma,\Sigma'\subseteq\mathbb{C}$ open such that $\xi\left(\Sigma\right)=\Sigma'$.
If $h\in\mathcal{D}_{\Sigma'}\left(X\right)$ then $\left(s,x\right)\longmapsto h\left(\xi\left(s\right),x\right)\in\mathcal{D}_{\Sigma}\left(X\right)$.
\item $\mathcal{D}$ and $\mathcal{C}^{\mathbb{K}}$ are stable under right-composition
with subanalytic maps, in the following sense. Let $X\subseteq\mathbb{R}^{m},Y\subseteq\mathbb{R}^{n}$
be subanalytic and $\varphi:X\longrightarrow Y$ be a map with components
in $\mathcal{S}\left(X\right)$. If $h\in\mathcal{D}\left(Y\right)$
and $g\in\mathcal{C}^{\mathbb{K}}\left(Y\right)$ then $\left(s,x\right)\longmapsto h\left(s,\varphi\left(x\right)\right)\in\mathcal{D}\left(X\right)$
and $g\circ\varphi\in\mathcal{C}^{\mathbb{K}}\left(X\right)$.
\end{enumerate}
Finally, for $h\in\mathcal{D}\left(X\times\mathbb{R}\right)$ without
poles outside some closed discrete set $P\subseteq\mathbb{K}$, we
describe (uniformly in the parameters $\left(s,x\right)$) the behaviour
of $h$ when $y\longrightarrow+\infty$. For this, we apply Proposition
\ref{prop: prep of C^M} to prepare $h$ and we concentrate on the
unique cell $A$ (with base $X$) which has vertical unbounded fibres.
By Remark \ref{rem: cell at infty}, $\Pi_{A}$ is the identity and
$A=B_{A}=\left\{ \left(x,y\right):\ x\in X,\ y>a\left(x\right)\right\} $. 
\end{rems}
Arguing as in the proof of Theorem \ref{thm: interpolation and locus C^M}
(the case of a cell with unbounded $y$-fibres) we can write, $\forall\left(s,x,y\right)\in\left(\mathbb{C}\setminus P\right)\times A$,
\[
h\left(s,x,y\right)=\sum_{i\leq N}T_{i}\left(s,x,y\right),
\]
where each $T_{i}$ is Puiseux in $y$, as in \eqref{eq: sum of gen with different puiseux data}.
Moreover, by enlarging $P$ to contain the \textquotedblleft collision
set\textquotedblright{} defined in \eqref{eq:collision set}, we may
suppose that $\forall s\in\mathbb{C}\setminus P$, the tuples
\begin{equation}
\left(\frac{\ell_{i}s+\eta_{i}-k}{d},\mu_{i}\right)\ \ \ \ i\leq N,\ k\in\mathbb{N}\label{eq:exponents}
\end{equation}
are pairwise distinct. Recall that $\ell_{i},\eta_{i}\in\mathbb{K}$
and $d,\mu_{i}\in\mathbb{N}$.

Fix an enumeration $\mathbb{N}\ni j\longmapsto\left(i\left(j\right),k\left(j\right)\right)\in\left\{ 0,\ldots,N\right\} \times\mathbb{N}$,
so that we may rewrite \eqref{eq:exponents} as
\[
\left(\lambda_{j}\left(s\right),\nu_{j}\right)=\left(\frac{\ell_{i\left(j\right)}s+\eta_{i\left(j\right)}-k\left(j\right)}{d},\mu_{i\left(j\right)}\right).
\]
Define
\[
a_{j}\left(s\right)=\Re\left(\lambda_{j}\left(s\right)\right)=\frac{\Re\left(\ell_{i\left(j\right)}s+\eta_{i\left(j\right)}\right)-k\left(j\right)}{d},\ b_{j}\left(s\right)=\Im\left(\lambda_{j}\left(s\right)\right)=\frac{\Im\left(\ell_{i\left(j\right)}s+\eta_{i\left(j\right)}\right)}{d}.
\]
Notice that $b_{j}\left(s\right)$ takes at most $N+1$ different
values, for every fixed $s$. Hence, we may write $h$ as the sum
of a uniformly summable family of functions as follows:
\begin{equation}
h\left(s,x,y\right)=\sum_{j}h_{j}\left(s,x\right)y^{a_{j}\left(s\right)+\text{i}b_{j}\left(s\right)}\left(\log y\right)^{\nu_{j}},\label{eq: expansion}
\end{equation}
where $h_{j}\in\mathcal{D}\left(X\right)$.

In a forthcoming paper, we will use \eqref{eq: expansion} to show
that $\mathcal{C}^{\mathbb{K}}$ is stable under taking pointwise
limits and that neither of the classes $\mathcal{C}^{\mathbb{K}},\mathcal{C}^{\mathbb{K},\mathcal{M}},\mathcal{C}^{\mathcal{P}\left(\mathbb{K}\right),\mathcal{M}}$
contains the Fourier transforms of all subanalytic functions.

\bibliographystyle{alpha}
\newcommand{\etalchar}[1]{$^{#1}$}
\def\cprime{$'$}

\end{document}